\documentclass[11pt]{amsart}%

\usepackage{amsmath}
\usepackage[cmtip,all]{xy}%
\usepackage{amssymb}%
\usepackage{amscd}%
\usepackage{color}%
\usepackage[active]{srcltx}
\usepackage{mathrsfs}%
\usepackage[all]{xy}%
\def\ol#1{\overline{#1}}% 		overline
\def\wh#1{\widehat{#1}}% 	wide hat
\def\wt#1{\widetilde{#1}}% 	wide tilde
\theoremstyle{plain}
    \newtheorem{theorem}{Theorem}[subsection]
    \newtheorem{theoremint}{Theorem}
    \newtheorem{proposition}[theorem]{Proposition}
    \newtheorem{lemma}[theorem]{Lemma}
    \newtheorem{corollary}[theorem]{Corollary}
\theoremstyle{definition}
    \newtheorem{definition}[theorem]{Definition}

    \newtheorem{remark}[theorem]{Remark}

\numberwithin{equation}{subsection}
\def\Alphabet{A,B,C,D,E,F,G,H,I,J,K,L,M,N,O,P,Q,R,S,T,U,V,W,X,Y,Z}%  Capitalized Alphabet
\def\alphabet{a,b,c,d,e,f,g,h,i,j,k,l,m,n,o,p,q,r,s,t,u,v,w,x,y,z}%	lowercase alphabet
\def\endpiece{xxx}%									marks end of list
\def\makeAlphabet[#1]{\expandafter\makeA#1,xxx,}%		Ex. \makeAlphabet[A,B]
\def\makealphabet[#1]{\expandafter\makea#1,xxx,}%		Ex. \makealphabet[c,d]
\def\makeA#1,{\def\temp{#1}\ifx\temp\endpiece\else%
\mkbb{#1}\mkfrak{#1}\mkbf{#1}\mkcal{#1}\mkscr{#1}\expandafter\makeA\fi}%
\def\makea#1,{\def\temp{#1}\ifx\temp\endpiece\else\mkfrak{#1}\mkbf{#1}\expandafter\makea\fi}%
\def\mkbb#1{\expandafter\def\csname bb#1\endcsname{\mathbb{#1}}}%   
\def\mkfrak#1{\expandafter\def\csname fr#1\endcsname{\mathfrak{#1}}}%    Define frak
\def\mkbf#1{\expandafter\def\csname b#1\endcsname{\mathbf{#1}}}%           Define bold letters
\def\mkcal#1{\expandafter\def\csname c#1\endcsname{\mathcal{#1}}}%       Define calligraphy
\def\mkscr#1{\expandafter\def\csname s#1\endcsname{\mathscr{#1}}}%       Define script
\def\makeop[#1]{\xmakeop#1,xxx,}%					Ex. \makeop[Hom,Spec]
\def\mkop#1{\expandafter\def\csname #1\endcsname{{\mathrm{#1}}}} % 
\def\xmakeop#1,{\def\temp{#1}\ifx\temp\endpiece\else\mkop{#1}\expandafter\xmakeop\fi}%
\makeAlphabet[\Alphabet]%				Define bb, frak, bf, cal for Capitalized Alphabet
\makealphabet[\alphabet]%  				Define frak and bf for uncapitalized alphabet
\makeop[Alt,End,Ext,Hom,Sym,Tor,Aut,Det,Ind]% 		Homs
\makeop[CH,Pic,div]% 					Chow groups
\makeop[Ker,Im,Coker,Coim,id,pr,proj,ker,coker,im]% 		Kernels
\makeop[Spec,Spf,Spm,Isom]% 			Spaces
\makeop[Re,Im,Min,Max]%
\makeop[dR,Nis,crys,cris,rig,syn,tor,et,Cont]% 		Cohomologies
\makeop[Gal,Res,ab,res,sgn,R,Per]%				Galois groups
\makeop[pol,Iw,cycl,Exp,Log,Der]%							Polylogarithm
\makeop[Var,an,Isoc,univ,Cusp]%
\makeop[Lie,coLie,MIC,DR,Gr,Ab,Cone,ord,Frob,N,aH]
\makeop[Meas,arith,mot,can,tors,conv,opp,mult,const,para,near,naive,cont,contr]
\makeop[Gl,Sl]
\makeop[Eis,tr,Fil,gr,sp,Tate,rk,fin,cl,pt,trace]
\makeop[Gl,Sl,GL,SL]

\def\R{\mathbb{R}}

\def\mfg{{\mathfrak g}}
\def\ari#1{\widehat{#1}}
\def\ch{{\rm ch}}
\def\bbCA{{\mathcal A}}

\def\ac1{{\widehat{\rm c}_1}}
\def\an{{\rm an}}

\def\M0{M^0}
\def\sLog{\sL\!\!\operatorname{og}}

\def\isom{\cong}

\def\prolim{\varprojlim}
\def\Isocda{\Isoc^{\kern-0.5mm\dagger}}

\def\isom{\cong}

\def\prolim{\varprojlim}

\def\sLog{\sL\!\!\operatorname{og}}

\def\cyc{{\rm cyc}}
\def\BC{{\rm BC}}
\def\Id{{\rm Id}}

\begin{document}
%--- The Title ------------------------------------------------------------------------------------------------------------------------
\title[Norm compatible elements]{Higher analytic torsion, polylogarithms and norm compatible elements on
abelian schemes}
\begin{abstract}
We give a simple axiomatic description of the degree $0$ part of the polylogarithm on abelian schemes 
and show that its realisation in analytic Deligne cohomology can be described in terms of 
the Bismut-K\"ohler higher analytic torsion form of the Poincar\'e bundle.
\end{abstract}
\author{Guido Kings }
\author{Damian R\"ossler}
\thanks{G.K. was partly  supported by the DFG via grant
SFB 1085 "Higher Invariants".}
\maketitle
\tableofcontents

%%%%%%%%%%%%%%%%%%%%%%%%%%%%%%%%%%%%%%%%%%%%%%%%%%%%%%%%%%%%%%%%%%%%%%
%
%
%
\section*{Introduction}
%
%
%
%%%%%%%%%%%%%%%%%%%%%%%%%%%%%%%%%%%%%%%%%%%%%%%%%%%%%%%%%%%%%%%%%%%%%%
Norm compatible units play an important role in the arithmetic 
of cyclotomic fields and elliptic curves, especially in the context
of Iwasawa theory and Euler systems. The norm compatible units
in these cases are related to rich classical objects like polylogarithm functions and Eisenstein series. 

The search for analogues in the case of abelian schemes was for
a long time obstructed by the fact that the focus was narrowed on units rather than on 
classes in other $K$-groups. If one generalizes the question to
classes in $K_1$, then two different constructions were given in \cite{Kings99} and
\cite{Loopaper}. The first construction in \cite{Kings99}
relies on the motivic polylogarithm and gives not only norm compatible classes but a whole
series of classes satisfying a distribution relation. The second 
approach in \cite{Loopaper} depends on the arithmetic Riemann-Roch theorem applied to 
the Poincar\'e bundle and constructs a certain current, which
turns out to be in the image of the regulator map from $K_1$. 

It is a natural guess that these two approaches should be related or
in fact more or less equivalent. In the following paper we show that this expectation is founded but in a rather ad hoc way. To explain our result,
let us fix some notation. Let $\pi:\cA\to S$ be an abelian scheme of relative dimension $g$, let $\epsilon:S\to \cA$ be the zero section,
$N>1$ an integer and let $\cA[N]$ be the finite group scheme of $N$-torsion points. Here $S$ is smooth over a subfield of the complex numbers. Then
the (zero step of the) motivic polylogarithm  is a class in motivic cohomology 
\[
\pol^{0}\in H^{2g-1}_\sM(\cA\backslash \cA[N],g).
\] 
To describe it more precisely, consider the residue map along $\cA[N]$
\[
H^{2g-1}_\sM(\cA\backslash \cA[N],g)\to H^{0}_\sM(\cA[N]\backslash \epsilon(S),0).
\]
If we denote by $(\cdot)^{(0)}$ the generalized eigenspace of $\tr_{[a]}$
with eigenvalue $a^{0}=1$, then the residue map becomes an isomorphism 
\[
H^{2g-1}_\sM(\cA\backslash \cA[N],g)^{(0)}\isom H^{0}_\sM(\cA[N]\backslash \epsilon(S),0)^{(0)}
\]
and $\pol^{0}$ is the unique element mapping to the fundamental class
of $\cA[N]\backslash \epsilon(S)$.

Similarly, the current $\mfg_{\cA^\vee}$ on $\cA$ constructed in \cite{Loopaper} gives rise
to a class in analytic Deligne cohomology
\[
([N]^*\mfg_{\cA^\vee}-N^{2g}\cdot \mfg_{\cA^\vee})|_{\cA\backslash\cA[N]}\in H^{2g-1}_{D,\an}((\cA\backslash \cA[N])_\R, \R(g)),
\]
which lies in the image of the regulator map
\[
\cyc_\an:H^{2g-1}_\sM(\cA\backslash \cA[N],g)\to 
H^{2g-1}_{D,\an}((\cA\backslash \cA[N])_\R, \R(g)).
\]
We prove: 
 \begin{theoremint}
 We have 
 $$-2\cdot \cyc_\an(\pol^0)=([N]^*\mfg_{\cA^\vee}-N^{2g}\cdot \mfg_{\cA^\vee})|_{\cA\backslash\cA[N]}.$$
 
 Furthermore, the map $$H^{2g-1}_{\mathcal M}(\cA\backslash\cA[N],g)^{(0)}\to 
 H^{2g-1}_{D,\an}((\cA\backslash\cA[N])_\bbR,\bbR(g))$$ induced by $\cyc_\an$ is injective.
 \label{polgg}
 \end{theoremint} 
 (Theorem \ref{polgg} is Theorem \ref{polg} below)
 
Unfortunately, the obvious strategy to prove this theorem (ie to show that the class in analytic Deligne cohomology described above satisfies the same axiomatic 
properties as the zero step of the polylogarithm) fails because analytic Deligne cohomology does not 
come with residue maps. 
Instead we proceed as follows. We first show that the theorem is true for an abelian scheme
if it is true for one of its closed fibres. After that, using the compatibility 
of both the polylogarithm and the current under base change we show that it suffices to consider the universal abelian scheme over
a suitable moduli space of abelian varieties. There we have a special fibre, which is a
product of elliptic curves, where the theorem can be checked by
direct computations.

\textbf{Acknowledgements:} The authors would like to thank
Fr\'ed\'eric D\'eglise for providing a useful reference on Gysin sequences.  
%%%%%%%%%%%%%%%%%%%%%%%%%%%%%%%%%%%%%%%%%%%%%%%%%%%%%%%%%%%%%%%%%%%%%%
%
%
%
\section{Notations}
%
%
%
%%%%%%%%%%%%%%%%%%%%%%%%%%%%%%%%%%%%%%%%%%%%%%%%%%%%%%%%%%%%%%%%%%%%%%
We fix a base scheme $S$, which is irreducible, smooth and quasi-projective
over a field. This condition is necessary to apply the results of
Deninger-Murre on the decomposition of the Chow-motive of an abelian scheme.
We will work with an abelian scheme $\pi:{\mathcal A}\to S$ 
of fixed relative dimension $g$ with
unit section $\epsilon:S\to {\mathcal A}$. For any variety over a field
Soul\'e \cite{Soule-Operations} has defined motivic cohomology and homology 
\begin{align*}
H^i_\sM(X,j)&:=\Gr_\gamma^jK_{2j-i}(X)\otimes \bbQ\\
H_i^\sM(X,j)&:=\Gr^\gamma_jK'_{i-2j}(X)\otimes \bbQ,
\end{align*}
which form a twisted Poincar\'e duality theory in the sense of
Bloch-Ogus. 
%%%%%%%%%%%%%%%%%%%%%%%%%%%%%%%%%%%%%%%%%%%%%%%%%%%%%%%%%%%%%%%%%%%%%%
%
%
%
\section{Norm compatible elements on abelian schemes}
%
%
%
%%%%%%%%%%%%%%%%%%%%%%%%%%%%%%%%%%%%%%%%%%%%%%%%%%%%%%%%%%%%%%%%%%%%%%
%%%%%%%%%%%%%%%%%%%%%%%%%%%%%%%%%%%%%%%%%%%%%%%%%%%%%%%%%%%%%%%%%%%%%%
%
\subsection{The trace operator}
%
%%%%%%%%%%%%%%%%%%%%%%%%%%%%%%%%%%%%%%%%%%%%%%%%%%%%%%%%%%%%%%%%%%%%%%
Fix an integer $a>1$ prime to the characteristic of the ground field and let
$\cB\to S$ be an abelian scheme. In the applications $\cB$ will
be ${\mathcal A}$ or ${\mathcal A}\times_S\cdots\times_S{\mathcal A}$. We consider open 
sub-schemes $W\subset \cB$ with the property that 
$$
j:[a]^{-1}(W)\subset W
$$
is an open immersion,
where $[a]:\cB\to\cB$ is the $a$-multiplication on $\cB$.
As $[a]$ is finite \'etale and $j$ an open immersion, pull-back
$j^*$ and push-forward $[a]_*$ are defined on motivic homology and
cohomology. 
\begin{definition} For open sub-schemes $W\subset \cB$ as above,
we define the \emph{trace map} with respect to $a$ by
$$
\tr_{[a]}:H^\cdot_\sM(W,*)\xrightarrow{j^*}H^\cdot_\sM([a]^{-1}(W),*)
\xrightarrow{[a]_*}H^\cdot_\sM(W,*)
$$
and
$$
\tr_{[a]}:H_\cdot^\sM(W,*)\xrightarrow{j^*}H_\cdot^\sM([a]^{-1}(W),*)
\xrightarrow{[a]_*}H_\cdot^\sM(W,*).
$$
For any integer $r$, we let
$$
H^\cdot_\sM(W,*)^{(r)}:=\{\xi\in H^\cdot_\sM(W,*)\mid (\tr_{[a]}-a^r\id)^k\xi=0
\mbox{ for some }k\ge 1 \}
$$
be the generalized eigenspace of $\tr_{[a]}$ of \emph{weight} $r$.
\label{deftra}
\end{definition}
\begin{remark}
Of course, one can use any twisted Poincar\'e theory in the sense
of Bloch-Ogus instead of motivic cohomology for the definition of
the trace operator and its properties established below. 
\end{remark}
\begin{lemma}
Suppose that $\tr_{[a]}$ and $\tr_{[b]}$ are defined on 
$H^\cdot_\sM(W,*)$ and $H_\cdot^\sM(W,*)$, then the actions
of $\tr_{[a]}$ and $\tr_{[b]}$ commute.
\end{lemma}
\begin{proof}
This follows immediately from the diagram
\[
\xymatrix{
H^\cdot_{\sM}([a]^{-1}W,*)\ar[r]^{[a]_*}\ar[d]& H^\cdot_{\sM}(W,*)\ar[d]^{[a]_*}\\
H^\cdot_{\sM}([ab]^{-1}W,*)\ar[r]^{[ab]_*}& H^\cdot_{\sM}(W,*)\\
H^\cdot_{\sM}([b]^{-1}W,*)\ar[r]^{[b]_*}\ar[u]& H^\cdot_{\sM}(W,*)\ar[u]_{[b]_*}
}
\]
and in exactly the same way for homology.
\end{proof}
The next lemma shows that the localization sequences are equivariant
for the $\tr_{[a]}$-action in many situations.
\begin{lemma}\label{tr-and-loc-seq}
Let $W\subset \cB$ be an open sub-scheme with $[a]^{-1}W\subset W$ and
$Z\subset W$ be a closed sub-scheme. 
 Define $\wt Z$ by the cartesian
diagram
$$
\xymatrix{
Z\ar@^{(->}[r]&W\\
\wt Z\ar@^{(->}[u]\ar@^{(->}[r]&[a]^{-1}W\ar@^{(->}[u]}
$$
and assume that $[a](\wt Z)=Z$. Then the localization sequence
$$
\to H^\sM_\cdot(Z,*)\to  H^\sM_\cdot(W,*)\to  H^\sM_\cdot(W\backslash Z,*)\to 
$$
is equivariant for the $\tr_{[a]}$-action. 
\end{lemma}
\begin{proof} 
As $W$ is smooth and irreducible, the localization sequence
in motivic homology is isomorphic to the localization sequence in
cohomology with supports. This localization sequence is functorial with respect to 
Cartesian diagrams of closed immersion by the usual Bloch-Ogus axioms, which leads to a commutative diagram
$$
\xymatrix{
\ar[r]&H^\cdot_{Z,\sM}(W,*)\ar[r]\ar[d]& H^\cdot_{\sM}(W,*)\ar[r]\ar[d]& H^\cdot_{\sM}(W\backslash Z,*)\ar[r]\ar[d]&\\
\ar[r]&H^\cdot_{\wt Z,\sM}([a]^{-1}W,*)\ar[r]\ar[d]_\isom& H^\cdot_{\sM}([a]^{-1}W,*)\ar[r]\ar[d]_\isom& H^\cdot_{\sM}([a]^{-1}W\backslash \wt Z,*)\ar[r]\ar[d]_\isom&\\
\ar[r]&H_{2g-\cdot}^{\sM}(\wt Z,*)\ar[r]& H_{2g-\cdot}^{\sM}([a]^{-1}W,*)\ar[r]& H_{2g-\cdot}^{\sM}([a]^{-1}W\backslash \wt Z,*)\ar[r].&
}
$$
If we combine this with the functoriality of the homology localization
sequence for the proper morphism $[a]$, we get the desired result.
\end{proof}
%%%%%%%%%%%%%%%%%%%%%%%%%%%%%%%%%%%%%%%%%%%%%%%%%%%%%%%%%%%%%%%%%%%%%%
%
\subsection{Consequences of the motivic decomposition of the diagonal}
%
%%%%%%%%%%%%%%%%%%%%%%%%%%%%%%%%%%%%%%%%%%%%%%%%%%%%%%%%%%%%%%%%%%%%%%
Let $\Delta\subset {\mathcal A}\times_S{\mathcal A}$ be the relative diagonal and
$$
\cl(\Delta)\in H_\sM^{2g}({\mathcal A}\times_S{\mathcal A}, g)
$$
its class in motivic cohomology. The main result by Deninger and Murre
in \cite[Theorem 3.1]{De-Mu} implies that there is a decomposition 
\begin{equation}
\cl(\Delta)=\sum_{i=0}^{2g}\pi_i\in H_\sM^{2g}({\mathcal A}\times_S{\mathcal A}, g),
\end{equation}
where the $\pi_i$ are idempotents for the composition of correspondences
and mutually commute. The $\pi_i$ have the important property that
for all integers $a$
$$
[a]_*\pi_i=a^{2g-i}\pi_i.
$$
\begin{proposition}\label{decomp-prop}
There is a decomposition into $\tr_{[a]}$-eigenspaces
$$
H^\cdot_\sM({\mathcal A},*)\isom \bigoplus_{r=0}^{2g}H^\cdot_\sM({\mathcal A},*)^{(r)},
$$
which is independent of $a$.
Moreover, 
$$
H^\cdot_\sM({\mathcal A}\backslash \epsilon(S),*)^{(0)}=0.
$$
\end{proposition}
\begin{proof}
The first statement follows from the decomposition
$$
H^\cdot_\sM({\mathcal A},*)\isom \bigoplus_{i=0}^{2g}\pi_i H^\cdot_\sM({\mathcal A},*),
$$
which is independent of $a$. The second statement follows 
from the fact that
$\pi_{2g}H^\cdot_\sM({\mathcal A},*)\isom \epsilon_*H^{\cdot-2g}_\sM(S,*-g)$ and the equivariance of
the localization sequence
$$
\cdots\to H^{\cdot-2g}_\sM(S,*-g)\xrightarrow{\epsilon_*}
H^\cdot_\sM({\mathcal A},*)\to H^\cdot_\sM({\mathcal A}\backslash \epsilon(S),*)\to \cdots.
$$
\end{proof}
The following simple corollary is basic for everything that follows:
\begin{corollary}
Let $N\ge 2$ and ${\mathcal A}[N]\subset {\mathcal A}$ be the sub-scheme of $N$-torsion points,
then 
$$
H^{2g-1}_\sM({\mathcal A}\backslash{\mathcal A}[N],g)^{(0)}\isom
H^0_\sM({\mathcal A}[N]\backslash\epsilon(S),0)^{(0)}.
$$
\label{corfund}
\end{corollary}
\begin{proof}
This is a direct consequence of Proposition \ref{decomp-prop},
the localization sequence
$$
H^{2g-1}_\sM({\mathcal A}\backslash\epsilon(S),g)\to H^{2g-1}_\sM({\mathcal A}\backslash {\mathcal A}[N],g)
\to H^0_\sM({\mathcal A}[N]\backslash\epsilon(S),0)\to H^{2g}_\sM({\mathcal A}\backslash\epsilon(S),g)
$$
and Lemma \ref{tr-and-loc-seq}.
\end{proof}
Note that $H^0_\sM({\mathcal A}[N]\backslash\epsilon(S),0)^{(0)}$ is not zero.
It contains the fundamental class of $ {\mathcal A}[N]\backslash\epsilon(S)$.

%%%%%%%%%%%%%%%%%%%%%%%%%%%%%%%%%%%%%%%%%%%%%%%%%%%%%%%%%%%%%%%%%%%%%%
%
%
%
\section{The polylogarithm on abelian schemes}
%
%
%
%%%%%%%%%%%%%%%%%%%%%%%%%%%%%%%%%%%%%%%%%%%%%%%%%%%%%%%%%%%%%%%%%%%%%%

%%%%%%%%%%%%%%%%%%%%%%%%%%%%%%%%%%%%%%%%%%%%%%%%%%%%%%%%%%%%%%%%%%%%%%
%
\subsection{A motivation from topology}
%
%%%%%%%%%%%%%%%%%%%%%%%%%%%%%%%%%%%%%%%%%%%%%%%%%%%%%%%%%%%%%%%%%%%%%%
In this section we explain the topological polylogarithm. For
more details and applications consult \cite{BKL}. 

Let $X:=\bbC^g/\Lambda$ a complex torus and consider the group ring
$\bbZ[\Gamma]$ with its standard action of $\Gamma$ by
$\gamma(\gamma'):=(\gamma+\gamma')$. Let $I$ be the augmentation ideal for
$\bbZ[\Gamma]\to\bbZ$, $(\gamma)\mapsto 1$. Define
$$
\bbZ[[\Gamma]]:=\prolim_n\bbZ[\Gamma]/I^{n+1}
$$
and observe that $I^n/I^{n+1}\isom \Sym^n\Gamma$. In fact, choosing
a basis for $\Gamma$, one has 
$\bbZ[\Gamma]\isom \bbZ[t_1,t_1^{-1},\ldots,t_{2g},t_{2g}^{-1}]$ and
$I$ is the ideal $(t_1-1,\ldots,t_{2g}-1)$. The action of
$\Gamma$ on $\bbZ[\Gamma]$ extends to $\bbZ[[\Gamma]]$ and we
denote by 
$$
\sLog_\bbZ
$$
the sheaf on $X$ associated with the $\Gamma$-module $\bbZ[[\Gamma]]$.
We also write 
$$
\sLog_\bbZ^{(n)}
$$
for the sheaf associated with the $\Gamma$-module $\bbZ[[\Gamma]]/\wh I^{n+1}$,
where $\wh I$ is the augmentation ideal of $\bbZ[[\Gamma]]$. One gets
that $\sLog_\bbZ\isom \prolim_n\sLog_\bbZ^{(n)}$.
Another description of $\sLog_\bbZ$ is as follows: Let $p:\bbC^g\to X$
be the universal covering and consider $p_!\bbZ$ on $X$. This sheaf is a
$\pi^*\bbZ[\Gamma]$-module, where $\pi:X\to \pt$ is the structure map
of $X$. Then
\begin{equation}
\sLog_\bbZ\isom p_!\bbZ\otimes_{ \pi^*\bbZ[\Gamma]}\pi^*\bbZ[[\Gamma]].
\end{equation}
From this description, one sees without any effort that
\begin{equation}
R^i\pi_*\sLog_\bbZ\isom \begin{cases}
0&\mbox{ if }i\neq 2g\\
\bbZ&\mbox{ if }i=2g,
\end{cases}
\end{equation}
because 
$$
R^i\pi_*\sLog_\bbZ\isom 
H^i_c(X,p_!\bbZ)\otimes_{ \pi^*\bbZ[\Gamma]}\pi^*\bbZ[[\Gamma]]\isom 
H^i_c(\bbC^g,\bbZ)\otimes_{ \pi^*\bbZ[\Gamma]}\pi^*\bbZ[[\Gamma]]
$$
as $\bbZ[[\Gamma]]$ is a flat $\bbZ[\Gamma]$-module. Let $\epsilon:\pt\to X$ 
be the zero section, then the exact triangle
$$
\epsilon_*\epsilon^!\sLog_\bbZ\to\sLog_\bbZ\to Rj_*j^*\sLog_\bbZ
$$
where $j:X\backslash 0\to X$ induces
$$
\Ext^{2g-1}_{X}(\pi^*\Gamma, \sLog_\bbZ)\to 
\Ext^{2g-1}_{X\backslash 0}(\pi^*\Gamma, \sLog_\bbZ)\to
\Hom(\Gamma,I)\to \Ext^{2g}_{X}(\pi^*\Gamma, \sLog_\bbZ)
$$
where the $\Ext$-groups are extension classes of local systems and
$\pi^*\Gamma$ is considered as trivial local system. From the above
cohomology computation it follows that $\Ext^{2g-1}_{X}(\pi^*\Gamma, \sLog_\bbZ)=0$
and that $\Ext^{2g}_{X}(\pi^*\Gamma, \sLog_\bbZ)\isom \Hom(\Gamma,\bbZ)$.
We get 
$$
0\to 
\Ext^{2g-1}_{X\backslash 0}(\pi^*\Gamma, \sLog_\bbZ)\to
\Hom(\Gamma,\bbZ[[\Gamma]])\to \Hom(\Gamma,\bbZ)
$$
and the last map is easily seen to be induced by the augmentation of
$\bbZ[[\Gamma]]$. Thus, we have an isomorphism
\begin{equation}
\Ext^{2g-1}_{X\backslash 0}(\pi^*\Gamma, \sLog_\bbZ)\isom 
\Hom(\Gamma,\wh I).
\end{equation}
\begin{definition}
The \emph{(topological) polylogarithm} is the class 
$$
\pol\in \Ext^{2g-1}_{X\backslash 0}(\pi^*\Gamma, \sLog_\bbZ)
$$
that maps to the canonical inclusion $\Gamma\subset \wh I$ under
the above isomorphism.
\end{definition}

%%%%%%%%%%%%%%%%%%%%%%%%%%%%%%%%%%%%%%%%%%%%%%%%%%%%%%%%%%%%%%%%%%%%%%
%
\subsection{Review of the logarithm sheaf}
%
%%%%%%%%%%%%%%%%%%%%%%%%%%%%%%%%%%%%%%%%%%%%%%%%%%%%%%%%%%%%%%%%%%%%%%
\label{log-sheaf-section}
Inspired by the above topological construction, one can define
a logarithm sheaf in any reasonable sheaf theory, most notably for
\'etale sheaves or Hodge-modules. This construction was first
carried out in \cite{w} and generalizes the case of elliptic curves
from \cite{BeLe}. The construction is very formal
and does not need specific properties of the respective sheaf theory.

As in our main reference \cite{Kings99}, we want to define the
logarithm sheaf at the same time for $\ell$-adic sheaves and for
Hodge-modules over $\bbR$. In the Hodge-module setting we let $F=\bbR$
and all varieties are considered over $\bbR$. Lisse sheaves are the
ones where the underlying perverse sheaf is a local system placed in
degree $[-\mbox{dimension of the variety}]$. We will consider these
sheaves as sitting in degree $0$ to have an easier comparison with the
\'etale situation. In the $\ell$-adic setting we let $F=\bbQ_\ell$
and a lisse sheaf is associated with a continuous representation of the
\'etale fundamental group. 

\begin{definition}
Let $\sH:=(R^1\pi_*F)^\vee=\Hom_S(R^1\pi_*F,F)$ be the dual of the
first relative cohomology of $\pi:{\mathcal A}\to S$.
\end{definition}
The Leray spectral sequence induces a short exact sequence
$$
0\to \Ext^1_S(F,\sH)\xrightarrow{\pi^*}\Ext^1_{\mathcal A}(F,\pi^*\sH)\to 
\Hom_S(F,R^1\pi_*F\otimes \sH)\to 0,
$$
which is exact and split because $\pi$ has the section $\epsilon:S\to {\mathcal A}$.
Note that 
$$
\Hom_S(F,R^1\pi_*F\otimes \sH)\isom \Hom_S(\sH,\sH).
$$
\begin{definition}
Let $\sLog^{(1)}\in \Ext^1_{\mathcal A}(F,\pi^*\sH)$ be the unique class, which maps
to $\id_\sH\in \Hom_S(\sH,\sH)$ and such that $\epsilon^*\sLog^{(1)}$ splits. 
We write also $\sLog^{(1)}$ for the lisse sheaf representing this class. 
\end{definition}
By definition $\sLog^{(1)}$ sits in an exact sequence
$$
0\to \pi^*\sH\to \sLog^{(1)}\to F\to 0.
$$
We define
$$
\sLog^{(n)}:=\Sym^n\sLog^{(1)}
$$
so that we have morphisms $\sLog^{(n)}\to \sLog^{(n-1)}$ induced
by $\sLog^{(1)}\to F$. We write $\sLog$ for the projective system
$(\sLog^{(n)})_n$. 
Note that $\sLog$ is a pro-unipotent sheaf, which is a successive extension
of $\Sym^k\sH$.
 
Let $\psi:{\mathcal A}\to \cB$ be an isogeny. Then $\psi$ induces an isomorphism
$\sH_{\mathcal A}\isom \psi^*\sH_\cB$ and hence an isomorphism 
$\sLog_{\mathcal A}^{(1)}\isom \psi^*\sLog^{(1)}_\cB$ or more generally
\begin{equation}\label{log-splitting}
\sLog_{\mathcal A}\isom \psi^*\sLog_\cB.
\end{equation}
For $t\in\ker\psi(S)$, we get $t^*\sLog_{\mathcal A}\isom \epsilon_{\mathcal A}^*\sLog_\cB$,
which induces an isomorphism
\[
t^*\sLog_{\mathcal A}\isom \epsilon_{\mathcal A}^*\sLog_{\mathcal A}\isom \prod_{n\ge 0}\Sym^n\sH_{\mathcal A}.
\]
As in the topological case one can compute
the cohomology of $\sLog$.
\begin{proposition}[\cite{Kings99}, Proposition 1.1.3]\label{log-coh-comp} For the higher direct images of $\sLog^{(n)}$ one has
$$
R^{2g}\pi_*\sLog^{(n)}\isom R^{2g}\pi_*\sLog^{(n-1)}\isom \cdots\isom R^{2g}\pi_*F
\isom F(-g)
$$
and for $i<2g$ the maps $\sLog^{(n)}\to \sLog^{(n-1)}$ induce the zero map
$$
R^{i}\pi_*\sLog^{(n)}\isom R^{i}\pi_*\sLog^{(n-1)}
$$
for all $n$. 
\end{proposition}
Let us define 
$$
H^j({\mathcal A},\sLog(g)):=\prolim_nH^j({\mathcal A},\sLog^{(n)}(g)),
$$
then the Proposition \ref{log-coh-comp} 
and the Leray spectral sequence for $R\pi_*$ imply that
\begin{equation}\label{log-coh-consequence}
H^j({\mathcal A},\sLog(g))\isom \begin{cases}
0&\mbox{ if }j<2g\\
H^0(S,F)&\mbox{ if }j=2g.
\end{cases} 
\end{equation}

%%%%%%%%%%%%%%%%%%%%%%%%%%%%%%%%%%%%%%%%%%%%%%%%%%%%%%%%%%%%%%%%%%%%%%
%
\subsection{The polylogarithm}
%
%%%%%%%%%%%%%%%%%%%%%%%%%%%%%%%%%%%%%%%%%%%%%%%%%%%%%%%%%%%%%%%%%%%%%%
Consider the exact triangle for the open immersion
$j:{\mathcal A}\backslash{\mathcal A}[N]\hookrightarrow {\mathcal A}$
and $\iota:{\mathcal A}[N]\hookrightarrow {\mathcal A}$
$$
\iota_*\iota^!\sLog\to\sLog\to Rj_*j^*\sLog,
$$
then we get a localization sequence
$$
H^{2g-1}({\mathcal A},\sLog(g))\to H^{2g-1}({\mathcal A}\backslash {\mathcal A}[N],\sLog(g))\to
H^{0}({\mathcal A}[N],\iota^*\sLog)\to H^{2g}({\mathcal A},\sLog(g)).
$$
\begin{corollary}\label{pol-ext}
Let $\sLog[{\mathcal A}[N]]^0:=\ker(\iota_*\iota^*\sLog\to F)$ be the kernel of the map
induced by the augmentation $\sLog\to F$. Then the localization sequence
induces an isomorphism
$$
H^{2g-1}({\mathcal A}\backslash {\mathcal A}[N],\sLog(g))\isom 
H^{0}(S,\sLog[A[N]]^{0}).
$$
\end{corollary}
\begin{proof}
From Equation \eqref{log-coh-consequence} we get that the localization sequence
gives 
$$
0\to H^{2g-1}({\mathcal A}\backslash {\mathcal A}[N],\sLog(g))\to
H^{0}({\mathcal A}[N],\iota^*\sLog)\to H^{0}(S,F)
$$
and the last map is induced by the augmentation $\sLog\to F$.
\end{proof}
Denote by $H^0({\mathcal A}[N],F)^0$ the kernel of the trace map
$$
H^0({\mathcal A}[N],F)^0:=\ker(H^0({\mathcal A}[N],F)\to H^0(S,F)).
$$
By Equation \eqref{log-splitting} we have an inclusion
$$
H^0({\mathcal A}[N],F)^0\subset  H^0(S,\sLog[{\mathcal A}[N]]^0).
$$
It is convenient to identify
\begin{equation}
H^0({\mathcal A}[N],F)^0\isom H^0({\mathcal A}[N]\backslash \epsilon(S),F)
\end{equation}
via the restriction map. 
\begin{definition}\label{pol-def}
For each $\varphi\in  H^0({\mathcal A}[N]\backslash \epsilon(S),F)$ we let 
$$
\pol_\varphi\in H^{2g-1}({\mathcal A}\backslash {\mathcal A}[N],\sLog(g))
$$
be the class, which maps to $\varphi$ under the isomorphism 
in Corollary \ref{pol-ext}. We let 
$$
\pol_\varphi^{n}\in H^{2g-1}({\mathcal A}\backslash {\mathcal A}[N],\sLog^{(n)}(g))
$$
be the image under the canonical map $\sLog\to \sLog^{(n)}$.
\end{definition}
If we want to specify the
theory of sheaves we working with, we write
$$
\pol_{\cH,\varphi}\in H^{2g-1}_\cH({\mathcal A}\backslash {\mathcal A}[N],\sLog(g))
$$
for the absolute Hodge and 
$$
\pol_{\et,\varphi}\in H^{2g-1}_{\et}({\mathcal A}\backslash {\mathcal A}[N],\sLog(g))
$$
for the $\ell$-adic polylogarithm.

%%%%%%%%%%%%%%%%%%%%%%%%%%%%%%%%%%%%%%%%%%%%%%%%%%%%%%%%%%%%%%%%%%%%%%
%
\subsection{Norm compatibility of the polylogarithm}
%
%%%%%%%%%%%%%%%%%%%%%%%%%%%%%%%%%%%%%%%%%%%%%%%%%%%%%%%%%%%%%%%%%%%%%%
Let $a\ge 2$ be an integer and consider the $[a]$-multiplication 
$[a]:{\mathcal A}\to{\mathcal A}$. We define an endomorphism $\tr_{[a]}$ of
$H^{2g-1}({\mathcal A}\backslash {\mathcal A}[N],\sLog(g))$ as follows:
\begin{multline*}
\tr_{[a]}:H^{2g-1}({\mathcal A}\backslash {\mathcal A}[N],\sLog(g))\to 
H^{2g-1}([a]^{-1}({\mathcal A}\backslash {\mathcal A}[N]),\sLog(g))\to \\
\to 
H^{2g-1}([a]^{-1}({\mathcal A}\backslash {\mathcal A}[N]),[a]^*\sLog(g))\xrightarrow{\trace_{[a]}}
H^{2g-1}({\mathcal A}\backslash {\mathcal A}[N],\sLog(g)),
\end{multline*}
where $\trace_{[a]}$ is the trace map relative to the finite morphism $[a]$.
\begin{proposition}\label{pol-norm-comp}
Suppose that $(a,N)=1$ and let $[a]_*\varphi$ be the image of
$\varphi\in  H^0({\mathcal A}[N]\backslash \epsilon(S),F)$ under the finite map $[a]$.
Then one has
$$
\tr_{[a]}\pol_\varphi=\pol_{[a]\varphi}.
$$
In particular, for $a\equiv 1\mod{N}$ one gets 
$\tr_{[a]}\pol_\varphi=\pol_{\varphi}$ and $\pol_\varphi$ is norm-compatible.
\end{proposition}
\begin{proof}
As the trace map is compatible with residues, we have a commutative diagram
$$
\xymatrix{
H^{2g-1}({\mathcal A}\backslash {\mathcal A}[N],\sLog(g))\ar[r]^/.5em/\res\ar[d]_{\tr_{[a]}}&
H^0(S,\sLog[{\mathcal A}[N]]^0)\ar[d]^{\tr_{[a]}}\\
H^{2g-1}({\mathcal A}\backslash {\mathcal A}[N],\sLog(g))\ar[r]^/.5em/\res&
H^0(S,\sLog[{\mathcal A}[N]]^0).
}
$$
The map ${\tr_{[a]}}$ induces on $H^0({\mathcal A}[N],F)^0\subset  H^0(S,\sLog[{\mathcal A}[N]]^0)$
the map $[a]_*:H^0({\mathcal A}[N],F)^0\to H^0({\mathcal A}[N],F)^0$. The result follows from
the definition of the polylogarithm.
\end{proof}
\begin{corollary}
Let $\pol^0_\varphi\in H^{2g-1}({\mathcal A}\backslash {\mathcal A}[N],F(g))$ be the
image of $\pol_\varphi$ under the morphism 
$$
H^{2g-1}({\mathcal A}\backslash {\mathcal A}[N],\sLog(g))\to H^{2g-1}({\mathcal A}\backslash {\mathcal A}[N],F(g))
$$
induced by the augmentation $\sLog\to F$. Then 
$$
\pol^0_\varphi\in H^{2g-1}({\mathcal A}\backslash {\mathcal A}[N],F(g))^{(0)}
$$
is in the generalized $0$ eigenspace of $\tr_{[a]}$.
\end{corollary}
\begin{proof}
This is clear from Proposition \ref{pol-norm-comp}.
\end{proof}
%%%%%%%%%%%%%%%%%%%%%%%%%%%%%%%%%%%%%%%%%%%%%%%%%%%%%%%%%%%%%%%%%%%%%%
%
\subsection{A motivic construction of the polylogarithm}
%
%%%%%%%%%%%%%%%%%%%%%%%%%%%%%%%%%%%%%%%%%%%%%%%%%%%%%%%%%%%%%%%%%%%%%%
The motivic construction presented here is similar to the one in
\cite{Kings99}, except that we consider here the variant of the polylogarithm
explained in \ref{pol-def}. We will focus on the differences and refer to
\cite{Kings99} whenever possible.

\begin{remark}
All constructions in this section work without any changes in any twisted Bloch-Ogus cohomology theory. To fix ideas, and because it is the "universal" case, we wrote everything for motivic cohomology.
\end{remark}

Let us recall some notations from \cite{Kings99}. Define
$$
U:=({\mathcal A}\backslash\epsilon(S))\times_S({\mathcal A}\backslash\epsilon(S))
$$
considered with the second projection $p:U\to {\mathcal A}\backslash\epsilon(S)$
as a scheme over ${\mathcal A}\backslash\epsilon(S)$. Let 
$V:=U\backslash\Delta$ be the complement of the diagonal and set
\begin{align*}
U^n&:=U\times_{{\mathcal A}\backslash\epsilon(S)}\cdots\times_{{\mathcal A}\backslash\epsilon(S)}U
\mbox{ $n$-times}\\
V^n&:=V\times_{{\mathcal A}\backslash\epsilon(S)}\cdots\times_{{\mathcal A}\backslash\epsilon(S)}V
\mbox{ $n$-times}.
\end{align*}
More generally, we let for $I\subset \{1,\ldots,n\}$ 
$$
V^I:=\{(u_1,\ldots,u_n)\in U^n\mid u_i\in V \mbox{ if }i\in I\mbox{ and }
u_i\in \Delta\mbox{ if }i\notin I\}.
$$
This gives a stratification $U^n=\bigcup_{I}V^I$.
Denote by $\Sigma_n$ the permutation group of $\{1,\ldots,n\}$
and let $\sgn_n$ denote the sign character. For any $\bbQ$-vector space $H$ 
with $\Sigma_n$ action, we denote by $H_{\sgn_n}$ the $\sgn_n$ eigenspace.

The fundamental result for the construction is:
\begin{proposition}\label{prop:localization-seq}[\cite{Kings99} Corollary 2.1.4]\label{main-exact-seq}
There is a long exact sequence
$$
\to H^\cdot_\sM(U^n,*)_{\sgn_n}\to
H^\cdot_\sM(V^n,*)_{\sgn_n}\xrightarrow{\res}
H^{\cdot-2g+1}_\sM(V^{n-1},*-g)_{\sgn_{n-1}}\to,
$$
where the residue map is taken along the $n$-th variable
and which is equivariant for the $\tr _{[a]}$ action for any integer
$a$.
\end{proposition}
The schemes $U^n$ and $V^n$ are considered over ${\mathcal A}\backslash\epsilon(S)$
and we consider the base change to ${\mathcal A}\backslash{\mathcal A}[N]\subset {\mathcal A}\backslash\epsilon(S)$:
\begin{align*}
U^n_{{\mathcal A}\backslash{\mathcal A}[N]}&:=U^n\times_{{\mathcal A}\backslash\epsilon(S)}{{\mathcal A}\backslash{\mathcal A}[N]}\\
V^n_{{\mathcal A}\backslash{\mathcal A}[N]}&:=V^n\times_{{\mathcal A}\backslash\epsilon(S)}{{\mathcal A}\backslash{\mathcal A}[N]}.
\end{align*}
Note that the base change is compatible with the $\Sigma_n$ action, so
that the same proof as for \cite[Corollary 2.1.4]{Kings99} 
shows that there is also a long exact sequence
$$
\to H^\cdot_\sM(U^n_{{\mathcal A}\backslash{\mathcal A}[N]},*)_{\sgn_n}\to
H^\cdot_\sM(V^n_{{\mathcal A}\backslash{\mathcal A}[N]},*)_{\sgn_n}\xrightarrow{\res}
H^{\cdot-2g+1}_\sM(V^{n-1}_{{\mathcal A}\backslash{\mathcal A}[N]},*-g)_{\sgn_{n-1}}\to.
$$
This sequence is still equivariant for the $\tr_{[a]}$-action. Now we use the fact that
\begin{equation}
H^\cdot_\sM(U^n,*)^{(0)}=0,
\end{equation}
which is an easy consequence of Proposition \ref{decomp-prop} by induction
(see also \cite[Theorem 2.2.3]{Kings99}). Note that the vanishing 
does not depend on the $a$ chosen to define  $\tr_{[a]}$.
If we 
combine this with the long exact sequence in the proposition we get:
\begin{corollary}\label{residue-isom}
The residue maps induce isomorphisms
\begin{align*}
H^\cdot_\sM(V^n_{{\mathcal A}\backslash{\mathcal A}[N]},*)^{(0)}_{\sgn_n}&\xrightarrow{\isom}
H^{\cdot-2g+1}_\sM(V^{n-1}_{{\mathcal A}\backslash{\mathcal A}[N]},*-g)^{(0)}_{\sgn_{n-1}}\xrightarrow{\isom}\cdots\\
&\cdots\xrightarrow{\isom}
H^{\cdot-(n+1)(2g-1)}_\sM({\mathcal A}[N]\backslash \epsilon(S),*-(n+1)g)^{(0)},
\end{align*}
which do not depend on the integer $a$ used to define the
operator $\tr _{[a]}$.
\end{corollary}
We can now define the motivic polylogarithm. 
\begin{definition}\label{mot-pol-def}
For any $\varphi\in H^0_\sM({\mathcal A}[N]\backslash \epsilon(S),0)^{(0)}$
define the \emph{motivic polylogarithm}
$$
\pol_{\sM,\varphi}^{n}\in H^{(2g-1)(n+1)}_\sM(V^n_{{\mathcal A}\backslash{\mathcal A}[N]},(n+1)g)^{(0)}_{\sgn_n}
$$
to be the class, which maps to $\varphi$ under the isomorphisms in
Corollary \ref{residue-isom}.
\end{definition}
From this characterization we get immediately:
\begin{proposition}
The polylogarithm is compatible with base change.
\end{proposition}
\begin{proof}
By \cite{Deglise-Exc} the Gysin sequence is compatible with base change.
\end{proof}
%%%%%%%%%%%%%%%%%%%%%%%%%%%%%%%%%%%%%%%%%%%%%%%%%%%%%%%%%%%%%%%%%%%%%%
%
\subsection{Comparison with the realizations of the polylogarithm}
%
%%%%%%%%%%%%%%%%%%%%%%%%%%%%%%%%%%%%%%%%%%%%%%%%%%%%%%%%%%%%%%%%%%%%%%
In this section we relate the motivic polylogarithm $\pol_{\sM,\varphi}$ 
from Definition \ref{mot-pol-def} via the regulator maps to
the $\pol_\varphi$ as in Definition \ref{pol-def}.

Consider the regulator maps
into absolute Hodge 
$$
r_\cH:H^{(2g-1)(n+1)}_\sM(V^n_{{\mathcal A}\backslash{\mathcal A}[N]},(n+1)g)\to
H^{(2g-1)(n+1)}_\cH(V^n_{{\mathcal A}\backslash{\mathcal A}[N]},(n+1)g)
$$
and to $\ell$-adic cohomology
$$
r_\ell:H^{(2g-1)(n+1)}_\sM(V^n_{{\mathcal A}\backslash{\mathcal A}[N]},(n+1)g)\to
H^{(2g-1)(n+1)}_\et(V^n_{{\mathcal A}\backslash{\mathcal A}[N]},(n+1)g).
$$
As in Section \ref{log-sheaf-section} we will treat the absolute Hodge
and the $\ell$-adic case simultaneously. We start 
by identifying $H^{(2g-1)(n+1)}(V^n_{{\mathcal A}\backslash{\mathcal A}[N]},(n+1)g)^{(0)}_{\sgn_n}$.
\begin{proposition}
One has an isomorphism
$$
H^{(2g-1)(n+1)}(V^n_{{\mathcal A}\backslash{\mathcal A}[N]},(n+1)g)^{(0)}_{\sgn_n}\isom
H^{2g-1}({\mathcal A}\backslash{\mathcal A}[N],\sLog^{(n)}(g))^{(0)}
$$
and a commutative diagram
$$
\xymatrix{
H^{(2g-1)(n+1)}(V^n_{{\mathcal A}\backslash{\mathcal A}[N]},(n+1)g)^{(0)}_{\sgn_n}\ar[r]^/.5em/\res\ar[d]_\isom&
H^{(2g-1)n}(V^{n-1}_{{\mathcal A}\backslash{\mathcal A}[N]},ng)^{(0)}_{\sgn_{n-1}}\ar[d]^\isom\\
H^{2g-1}({\mathcal A}\backslash{\mathcal A}[N],\sLog^{(n)}(g))^{(0)}\ar[r]&
H^{2g-1}({\mathcal A}\backslash{\mathcal A}[N],\sLog^{(n-1)}(g))^{(0)},
}
$$
where the lower horizontal map is induced by $\sLog^{(n)}\to \sLog^{(n-1)}$.
\end{proposition}
\begin{proof} This is essentially proven in \cite[Proposition 2.3.1]{Kings99}
but for the weight $1$ parts.

Let $\ol U:=({\mathcal A}\backslash \epsilon(S))\times {\mathcal A}$ and $\ol V:=\ol U\backslash\Delta$.
We denote by $\ol p$ the second projection. We also let
$\wt V^n:=V^n_{{\mathcal A}\backslash{\mathcal A}[N]}$ to shorten notation.
It is shown in the first part 
of the proof of \cite[Proposition 2.3.1]{Kings99} that
$$
R^i\ol p_*F_{\ol V}(g)\isom \begin{cases}
\pi^*R^i\pi_*F(g)&i<2g-1\\
\sLog^{(1)}&i=2g-1\\
0&i=2g.
\end{cases}
$$
It follows that $R^{2g-1}p_*F_{\wt V}$ is isomorphic to the restriction of
$\sLog^{(1)}$ to ${\mathcal A}\backslash {\mathcal A}[N]$.

Let $p_{V^n}:\wt V^n\to {{\mathcal A}\backslash{\mathcal A}[N]}$ be the 
structure map. As the $R^ip_{V^n,*}F_{\wt V^n}$ are
all lisse sheaves, we can take the dual of the K\"unneth formula for $Rp_{V^n,!}$
to get
$$
R^ip_{V^n,*}F_{\wt V^n}\isom
\bigoplus_{i_1+\cdots i_n=i}R^{i_1}p_*F_{\wt V}\otimes\cdots
\otimes R^{i_n}p_*F_{\wt V}.
$$
In particular, 
$R^{(2g-1)n}p_{V^n,*}F_{\wt V^n}\isom (\sLog^{(1)})^{\otimes n}$
on ${\mathcal A}\backslash{\mathcal A}[N]$.  Let $\wt\pi:{\mathcal A}\backslash{\mathcal A}[N]\to S$ be
the structure map. Note that the projection formula for $R\wt\pi_!$ gives
for lisse sheaves a projection formula for $R\wt\pi_*$. From the exact sequence 
$$
0\to \pi^*\sH\to \sLog^{(1)}\to F\to 0
$$
we see that $\tr_{[a]}$ acts with weights $\ge 0$ on $R^i\wt\pi_*\sLog^{(1)}$.
We claim that 
\begin{equation}\label{vanishing-claim}
H^\cdot({\mathcal A}\backslash {\mathcal A}[N],R^{i_1}p_*F_{\wt V}\otimes\cdots
\otimes R^{i_n}p_*F_{\wt V}(gn))^{(0)}=0
\end{equation}
whenever $i_1+\ldots +i_n<(2g-1)n$.
If this is the case, 
at least one factor in the tensor product is of
the form $\pi^*R^{i_j}\pi_*F$ with $i_j<2g-1$. Without loss of generality,
we may assume $j=1$. We get
\begin{multline*}
R^k\wt\pi_*(R^{i_1}p_*F_{\wt V}\otimes\cdots
\otimes R^{i_n}p_*F_{\wt V})\isom \\
\isom
R^{i_1} \pi_*F\otimes R^k\wt \pi_*(R^{i_2}p_*F_{\wt V}\otimes\cdots
\otimes R^{i_n}p_*F_{\wt V}).
\end{multline*}
The trace operator $\tr_{[a]}$ acts on $R^{i_1} \pi_*F$ with weight $\ge 2$
on $(R^{i_2}p_*F_{V_{{\mathcal A}\backslash{\mathcal A}[N]}}\otimes\cdots
\otimes R^{i_n}p_*F_{V_{{\mathcal A}\backslash{\mathcal A}[N]}})$ with some weight $\ge 0$.
This gives the vanishing in Equation \eqref{vanishing-claim}. We get
$$
H^{(2g-1)(n+1)}(\wt V^n,(n+1)g)^{(0)}_{\sgn_n}\isom
H^{2g-1}({\mathcal A}\backslash{\mathcal A}[N],\Sym^n\sLog^{(1)}),
$$
which gives the first claim of the proposition. The compatibility
of the residue with $\sLog^{(n)}\to \sLog^{(n-1)}$ follows from the
commutative diagram
$$
{\footnotesize
\xymatrix{
H^{(2g-1)(n+1)}(\wt U^n,(n+1)g)^{(0)}_{\sgn_n}\ar[r]\ar[d]^\isom&
H^{(2g-1)(n+1)}(\wt V^n,(n+1)g)^{(0)}_{\sgn_n}\ar[r]\ar[d]^\isom&
H^{(2g-1)n}(\wt V^{n-1},ng)^{(0)}_{\sgn_{n-1}}\ar[d]^=\\
H^{(2g-1)n}(\wt U^{n-1},p_{\wt U^{n-1}}^*\sH(ng)^{(0)}_{\sgn_n}\ar[r]&
H^{(2g-1)n}(\wt V^{n-1},p_{\wt V^{n-1}}^*\sLog^{(1)}(ng))^{(0)}_{\sgn_n}\ar[r]&
H^{(2g-1)n}(\wt V^{n-1},ng)^{(0)}_{\sgn_{n-1}}.
}}
$$
\end{proof}
\begin{corollary}
The motivic polylogarithm 
$$
\pol_{\sM,\varphi}^{n}\in H^{(2g-1)(n+1)}_\sM(V^n_{{\mathcal A}\backslash{\mathcal A}[N]},(n+1)g)^{(0)}_{\sgn_n}
$$
maps under the regulator maps to the absolute Hodge and to the \'etale polylogarithm
\begin{align*}
r_\sH(\pol_{\sM,\varphi}^{n})=\pol_{\cH,\varphi}^{n}&&
r_\et(\pol_{\sM,\varphi}^{n})=\pol_{\et,\varphi}^{n}.
\end{align*}
\end{corollary}

\section{The abelian polylogarithm in degree $0$ and the higher analytic torsion 
of the Poincar\'e bundle}

In the first subsection, we recall some concepts and results from Arakelov theory and we state the main result of 
this section, ie Theorem \ref{polg}. In the second subsection, we give a proof of Theorem \ref{polg}. 

\subsection{Canonical currents on abelian schemes}

We begin with a review of some notations and definitions coming from Arakelov theory.

Let $(R,\Sigma)$ be an arithmetic ring. By definition, this means
that $R$ is an excellent regular ring, which comes with a finite conjugation-invariant set $\Sigma$ of embeddings into $\bbC$ (see \cite[3.1.2]{Gillet-Soule-Arithmetic}).

For example $R$ might be $\bbZ$ with its unique embedding into $\bbC$, or $\bbC$ with
the identity and complex conjugation as embeddings.

An {\it arithmetic variety} over $R$ is a regular scheme, which is flat and quasi-projective over $R$. This definition is more restrictive than the definition of the same term given in 
\cite{Gillet-Soule-Arithmetic}.

For any arithmetic variety $X$ over $R$, we write as usual
$$
X(\bbC):=\coprod_{\sigma\in \Sigma}(X\times_{R,\sigma}\bbC)(\bbC)=:\coprod_{\sigma\in \Sigma}X(\bbC)_\sigma.
$$
For any $p\geqslant 0$, we let
$D^{p,p}(X_\bbR)$ (resp. $A^{p,p}(X_\bbR)$) be the $\bbR$-vector space of currents (resp. differential forms) $\gamma$ on $X(\bbC)$ such that
\begin{itemize}
\item[$\bullet$] $\gamma$ is a real current (resp. differential form) of type $(p,p)$;
\item[$\bullet$] $F_\infty^*\gamma=(-1)^p\gamma$,
\end{itemize}
where $F_\infty:X(\bbC)\to X(\bbC)$ is the real analytic involution given by complex conjugation. 
We then define 
$$
\widetilde{D}^{p,p}(X_\bbR):=D^{p,p}(X_\bbR)/(\im\ \partial+\im\ \bar\partial)\
$$
(resp.
$$
\widetilde{A}^{p,p}(X_\bbR):=A^{p,p}(X_\bbR)/(\im\ \partial+\im\ \bar\partial)\ \qquad\quad).
$$

All these notations are standard in Arakelov geometry. See \cite{Soule-Lectures} for
a compendium. It is shown in \cite[Th. 1.2.2 (ii)]{Gillet-Soule-Arithmetic} that the natural map $\widetilde{A}^{p,p}(X_\bbR)\to \widetilde{D}^{p,p}(X_\bbR)$ is an injection. 

If $Z$ a closed complex submanifold of $X(\bbC)$, we shall write 
more generally 
$D^{p,p}_Z(X_\bbR)$  for the $\bbR$-vector space of currents  $\gamma$ on $X(\bbC)$ such that
\begin{itemize}
\item[$\bullet$] $\gamma$ is a real current of type $(p,p)$;
\item[$\bullet$] $F_\infty^*\gamma=(-1)^p\gamma$;
\item[$\bullet$] the wave-front set of $\gamma$ is included in $N_{Z/X(\bbC)}\otimes_{\bbR}\bbC.$
\end{itemize}
Here $N$ is the conormal 
bundle of $Z$ in  $X(\bbC)$, where $Z$ and $X(\bbC)$ are viewed as real differentiable manifolds.

In the same way as above, we then define the $\bbR$-vector spaces 
% IMPORTANT %%%%%%%%%
% clarifiy (2i\pi) twist discrepancy
%%%%%%%%%%%%%%%%%%
$$
\widetilde{D}^{p,p}_Z(X_\bbR):=D^{p,p}_Z(X_\bbR)/D^{p,p}_{Z,0}(X_\bbR)
$$
where 
$D^{p,p}_{Z,0}(X_\bbR)$ is the set of currents $\gamma\in D^{p,p}_Z(X_\bbR)$ with the following property: there exists
a current $\alpha$ of type $(p-1,p)$ and a current $\beta$ of type $(p,p-1)$  such that $\gamma:=\partial\alpha+\bar\partial\beta$ and such that the 
wave-front sets of $\alpha$ and $\beta$ are included in the complexified conormal bundle of $Z$ in $X(\bbC)$.

See \cite{Hoermander-Analysis} for the definition (and theory) of the wave-front set.

It is a consequence of  \cite[Cor. 4.7]{Burgos-Litcanu-Singular} that the natural morphism $\widetilde{D}^{p,p}_Z(X_\bbR)\to \widetilde{D}^{p,p}(X_\bbR)$ is an injection.\footnote{many thanks to J.-I. Burgos for bringing this to our attention} 

Furthermore, it is a consequence of \cite[Th. 4.3]{Burgos-Litcanu-Singular} that for any $R$-morphism $f:Y\to X$ of arithmetic varieties, there is a natural morphism 
of $\bbR$-vector spaces $$f^*:\widetilde{D}^{p,p}_Z(X_\bbR)\to \widetilde{D}^{p,p}_{f(\bbC)^*(Z)}(Y_\bbR),$$ provided 
$f(\bbC)$ is transverse to $Z$. This morphism extends the morphism  $\widetilde{A}^{p,p}(X_\bbR)\to \widetilde{A}^{p,p}(Y_\bbR)$, 
which is obtained by pulling back differential forms.

We shall write $H^{*}_{D,\an}(X,\bbR(\cdot))$ for the analytic real Deligne cohomology of $X(\bbC)$.
By definition,
$$
H^{q}_{D,\an}(X,\bbR(p)) := {H}^q(X(\bbC),\bbR(p)_{D,\an})
$$
where $\bbR(p)_{D,\an}$ is the complex of sheaves of $\bbR$-vector spaces
$$
0\to\bbR(p)\to\Omega^1_{X(\bbC)}\stackrel{d\;}{\to}\Omega^1_{X(\bbC)}\to\dots\to\Omega^{p-1}_{X(\bbC)}\to 0
$$
on $X(\bbC)$ (for the ordinary topology). Here $\bbR(p):=(2i\pi)^p\,\bbR\subseteq\bbC$ and 
$\Omega^i_{X(\bbC)}$ is the sheaf of holomorphic forms of degree $i$. 
Notice that there is morphism of complexes $\bbR(p)_{D,\an}\to\bbR(p)$, 
 where $\bbR(p)$ is viewed as a complex of sheaves with a single entry in degree $0$. This morphism of 
 complexes induces maps 
 $$
\phi_B:H^{\ast}_{D,\an}(X,\bbR(\cdot))\to H^{\ast}(X(\bbC),\bbR(\cdot))
$$
from analytic Deligne cohomology into Betti cohomology. 
We now define
$$
H^{2p-1}_{D,\an}(X_\bbR,\bbR(p)):=\{\gamma\in H^{2p-1}_{D,\an}(X,\bbR(p))\;|\;F_\infty^*\gamma=(-1)^p\gamma\}.
$$
It is shown in \cite[par 6.1]{Burgos-Wang-Higher} that there is a natural inclusion
\begin{equation}
H^{2p-1}_{D,\an}(X_\bbR,\bbR(p))\hookrightarrow \widetilde{A}^{p,p}(X_\bbR).
\label{incl}
\end{equation}
There is a cycle class map $\cyc_\an:H^{*}_\sM(X,\bullet)\to H^{*}_{D,\an}(X_\bbR,\bbR(\bullet))$.  The composition $\phi_{\dR}\circ \cyc_\an$ is the 
usual cycle class map (or "regulator") into Betti cohomology. Finally there is a canonical exact sequence 
\begin{equation}
H^{2p-1}_\sM(\cA,p)_\bbQ
\xrightarrow{{\rm cyc}_\an}\widetilde{A}^{p-1,p-1}(\cA_\bbR){\to}\ari{\rm CH}^p(\cA)_\bbQ\to
{\rm CH}^p(\cA)_\bbQ\to 0
\label{arfundseq}
\end{equation}
where (abusing language) the map $\cyc_\an$ is defined via the inclusion \ref{incl}. The group \mbox{${\rm CH}^p(\cA)$} is the $p$-th Chow group 
of $X$ (see the book \cite{Fulton-Intersection}) and $\ari{\rm CH}^\bullet(\cdot)$ is the arithmetic Chow group of Gillet-Soul\'e (see \cite{Gillet-Soule-Arithmetic}). The group $\ari{\rm CH}^\bullet(\cdot)$ 
is contravariantly functorial for any $R$-morphisms of arithmetic 
varieties and covariantly functorial for smooth projective morphisms. 

Let now as before $\pi_\cA=\pi:{\mathcal A}\to S$ be an abelian scheme over $S$ of relative dimension $g=g_\cA$. We shall write
as usual ${\mathcal A}^\vee\to S$ for the dual abelian scheme. We let as before $\epsilon_\cA=\epsilon$ be the zero-section of ${\mathcal A}\to S$. We shall denote by $S_0=S_{0,\cA}=\epsilon_\cA(S)$ the (reduced) closed subscheme of 
$\bbCA$, which is the image of $\epsilon_\cA$. We write ${\mathcal P}$ for the Poincar\'e bundle on ${\mathcal A}\times_S{\mathcal A}^\vee$.
We endow ${\mathcal P}(\bbC)$ with the unique
metric $h_{\mathcal P}$ such that the canonical rigidification of ${\mathcal P}$ along the
zero-section ${\mathcal A}^\vee\to{\mathcal A}\times_S{\mathcal A}^\vee$ is an isometry and such that
the curvature form of $h_{\mathcal P}$ is translation invariant along the fibres of the map
\mbox{${\mathcal A}(\bbC)\times_{S(\bbC)}{\mathcal A}^\vee(\bbC)\to{\mathcal A}^\vee(\bbC)$.} We write $\overline{{\mathcal P}}:=({\mathcal P},h_{\mathcal P})$ for
the resulting hermitian line bundle. See \cite[chap I, 4 and chap. I]{Moret-Bailly-Pinceaux} 
for more details on all this. 

The following is Th. 1.1 in  \cite{Loopaper}.
\begin{theorem}[Maillot-R.]
There is a unique class of currents $\mfg_{{\mathcal A}^\vee}\in \widetilde{D}^{g-1,g-1}({\mathcal A}_\bbR)$ with the
following three properties:
\begin{itemize}
\item[\rm (a)] Any element of $\mfg_{{\mathcal A}^\vee}$ is a Green current for $S_0(\bbC)$.
\item[\rm (b)] The identity
$
(S_0,\mfg_{{\mathcal A}^\vee})=(-1)^g p_{1,*}(\ari{\ch}(\overline{{\mathcal P}}))^{(g)}
$
 holds in $\ari{\rm CH}^g({\mathcal A})_\bbQ$.
\item[\rm (c)]  The identity $\mfg_{{\mathcal A}^\vee}=[n]_*\mfg_{{\mathcal A}^\vee}$ holds for all $n\geqslant 2$.
\end{itemize}
\label{mainth0}
\end{theorem}
Here the morphism $p_1$ is the first projection ${\mathcal A}\times_S{\mathcal A}^\vee\to{\mathcal A}$ and
$[n]:{\mathcal A}\to{\mathcal A}$ is the multiplication-by-$n$ morphism. See \cite[1.2]{Gillet-Soule-Arithmetic} for the notion of Green current. 
The symbol $\ari{\ch}(\bullet)$ refers to the arithmetic Chern character, which has values in 
arithmetic Chow groups; see \cite{Gillet-Soule-Characteristic} for this. By $(\cdot)^{(g)}$ we mean the degree $g$ part of $(\cdot)$ in the 
natural grading of the arithmetic Chow group. 

In \cite{Loopaper} it is also shown that the restriction $(-1)^{g+1}\mfg_{{\mathcal A}^\vee}|_{\cA(\bbC)\backslash \epsilon(\bbC)}$ is equal to the part of degree $(g-1,g-1)$ of the Bismut-Koehler
 higher analytic torsion of $\overline{\mathcal P}$ along the map 
 $\cA(\bbC)\times\cA^\vee(\bbC)\to\cA(\bbC)$, for some natural choices of 
 K\"ahler fibration structures on $\cA(\bbC)\times\cA^\vee(\bbC)$. 

%Let now ${\bbCL}$ be a rigidified line bundle on $\bbCA^\vee$. Endow
%$\bbCL$ with the unique hermitian metric $h_\bbCL$, which is compatible with the rigidification and whose curvature form is translation-invariant on the fibres
%of $\bbCA^\vee(\bbC)\to S(\bbC)$. Let $\overline{\bbCL}:=(\bbCL,h_\bbCL)$ be the resulting hermitian line bundle. Let
%$\phi_\bbCL:\bbCA^\vee\to\bbCA$ be the polarisation morphism induced by $\bbCL$.
%
%\begin{theorem}
% Suppose that $\mathcal L$ is ample relatively to $S$ and symmetric.  Then the equalities
%$$
%(S_0,\mfg_{\bbCA^\vee})=(-1)^g p_{1,*}(\ari{\ch}(\overline{{\mathcal P}}))={1\over g!\sqrt{\deg(\phi_\bbCL)}}\phi_{{\mathcal L},*}(\ac1(\overline{{\mathcal L}})^g)
%$$
%are verified in $\ari{\rm CH}^g(\bbCA)_{\bbQ}$.
%\label{mth1}
%\end{theorem}

Furthermore (see \cite[Introduction]{Loopaper}),  the following is known:

- The class of currents $\mfg_{\cA^\vee}$ lies in $\widetilde{D}^{g-1,g-1}_{S_0(\bbC)}(\cA_\bbR)$. 

- Let $T$ be a an arithmetic variety over $R$ and $T\to S$ be a morphism of schemes over $R$. Let 
$\cA_T$ be the abelian scheme obtained by base-change and let ${\rm BC}:\cA_T\to \cA$ be the corresponding 
morphism. Then ${\rm BC}(\bbC)$ is tranverse to $S_0(\bbC)$ and ${\rm BC}^*\mfg^\vee_\cA=\mfg_{\cA_T}^\vee$.

%We shall write $H^q_D(X,\bbR(p))$ for the Deligne-Beilinson cohomology of $X(\bbC)$. See 
%\cite[chap. 10]{Burgos-The-regulators} for the definition. Furthermore, we define
%$$
%H^q_D(X_\bbR,\bbR(p)):=\{\gamma\in H^q_D(X(\bbC),\bbR(p))\;|\;F^{*}_{\infty}\gamma=(-1)^p\gamma\}.
%$$
%There is a natural arrow $H^q_D(X_\bbR,\bbR(p))\to H^q_{D,\an}(X_\bbR,\bbR(p))$, which 
%is an isomorphism if $X$ is proper over $k$. See \cite[chap. 10]{Burgos-The-regulators}. 
%
%Finally, there is a natural arrow $H^q_\cH(X,\bbR(p))\to H^q_{D}(X_\bbR,\bbR(p)).$ 
%See ... . 

We now suppose that the field $k$ is embeddable into $\bbC$ and we set $R=k$. 
We choose an arbitrary conjugation-invariant set of embeddings of $R$ into $\bbC$. 

 Fix once an for all $a\in\bbZ$ such that $(a,N)=1$.  
 Recall that by Corollary \ref{corfund}, we have an isomorphism 
\begin{equation}
\label{fundis}
\rho_\cA:H^{2g-1}_\sM({\mathcal A}\backslash{\mathcal A}[N],g)^{(0)}\isom
H^0_\sM({\mathcal A}[N]\backslash\epsilon_\cA(S),0)^{(0)}
\end{equation}
and that the element $\pol_{\sM,1_N^\circ,\cA,a}^0$ is the unique element of $H^{2g-1}_\sM({\mathcal A}\backslash{\mathcal A}[N],g)^{(0)}$ mapping 
to $1_N^\circ$ under this isomorphism. Recall that $1_N^\circ\in H^0_\sM({\mathcal A}[N]\backslash\epsilon_\cA(S),0)$ is the element given by 
the formal sum of all the irreducible components of ${\mathcal A}[N]\backslash\epsilon_\cA(S)$. 
Recall also that we have 
$$
H^{2g-1}_\sM({\mathcal A}\backslash{\mathcal A}[N],g)^{(0)}:=
\{h\in H^{2g-1}_\sM({\mathcal A}\backslash{\mathcal A}[N],g)\ |\ \exists l\geqslant 1:\ (\tr_{[a]}-\Id)^l(h)=0\}
$$
where $\tr_{[a]}$ is described in Definition \ref{deftra}, and
$$
H^0_\sM({\mathcal A}[N]\backslash\epsilon_\cA(S),0)^{(0)}:=
\{h\in H^0_\sM({\mathcal A}[N]\backslash\epsilon_\cA(S),0)\ |\ \ [a]_*(h)=h\}.
$$

 \begin{theorem}
 We have 
 $$-2\cdot \cyc_\an(\pol_{\sM,1_N^\circ,\cA,a}^0)=([N]^*\mfg_{\cA^\vee}-N^{2g}\cdot \mfg_{\cA^\vee})|_{\cA\backslash\cA[N]}.$$
 
 Furthermore, the map $$H^{2g-1}_{\mathcal M}(\cA\backslash\cA[N],g)^{(0)}\to 
 H^{2g-1}_{D,\an}((\cA\backslash\cA[N])_\bbR,\bbR(g))$$ induced by $\cyc_\an$ is injective.
 \label{polg}
 \end{theorem}
 
 (Theorem \ref{polg} is identical to Theorem \ref{polgg} in the introduction)
 
 The proof of Theorem \ref{polg} will occupy the next subsection.
 
 The proof goes as follows. We first prove the basic fact that  
 $$
 ([N]^*\mfg_{\cA^\vee}-N^{2g}\cdot \mfg_{\cA^\vee})|_{\cA\backslash\cA[N]}\in\cyc_\an(H^{2g-1}_{\mathcal M}(\cA\backslash\cA[N],g)^{(0)})
 $$
 (see Lemma \ref{eqal}). This is where Arakelov theory and the geometry of the Poincar\'e bundle play an important role.
 Next we show how the elements $\pol_{\sM,1_N^\circ,\cA,a}^0$ and $([N]^*\mfg_{\cA^\vee}-N^{2g}\cdot \mfg_{\cA^\vee})|_{\cA\backslash\cA[N]}$ 
 behave under products (see Lemmata \ref{mot-lem} and \ref{green-lemma}). To determine the behaviour of $\pol_{\sM,1_N^\circ,\cA,a}^0$ 
 under products, we need some elementary invariance results of residue maps in motivic cohomology and to determine the behaviour of $([N]^*\mfg_{\cA^\vee}-N^{2g}\cdot \mfg_{\cA^\vee})|_{\cA\backslash\cA[N]}$ under products, we need parts of the Gillet-Soul\'e calculus of Green currents.  
 In our third step, we prove that  
 Theorem \ref{polg} holds when $\cA$ is an elliptic curve; this follows from some classical results in the 
 theory of elliptic units (see Lemma \ref{ellpart}). In our fourth and final step, we show that Theorem \ref{polg} 
 holds for products of elliptic curves (see Lemma \ref{ellpart}) and we use a deformation argument 
 together with the existence of moduli spaces for polarised abelian varieties to reduce the general case 
 to the case of products of elliptic curves (see subsubsection \ref{ssspolg}). 
 
\begin{remark} (important) At first sight, it might seem that a natural way to tackle a statement like Theorem \ref{polg} is to check that the elements $-2\cdot \cyc_\an(\pol_{\sM,1_N^\circ,\cA,a}^0)$ and \mbox{$([N]^*\mfg_{\cA^\vee}-N^{2g}\cdot \mfg_{\cA^\vee})|_{\cA\backslash\cA[N]}$} have the same residue and are both norm invariant (ie $\tr_{[a]}$-invariant). One could then argue that a norm invariant element is completely 
determined by its residue to conclude. This line of proof does not work because there are no localisation sequences in 
analytic Deligne cohomology (unlike in Deligne-Beilinson cohomology). This is why we have to resort to the more complicated 
deformation argument outlined above, together with some explicit computations.
\end{remark}
 
 \subsection{Proof of Theorem \ref{polg}} 

We shall suppose without restriction of generality that $S$ is connected. We also suppose without restriction of generality that $a\equiv 1\ ({\rm mod}\ N)$. 
To see that this last hypothesis does not restrict generality, notice first that 
$\pol_{\sM,1_N^\circ,\cA,a}=\pol_{\sM,1_N^\circ,\cA,a^l}$ for all $l\geqslant 1$. Thus we may replace 
$a$ by some $a^l$ and since $(\bbZ/N\bbZ)^*$ is finite there exists $l\geqslant 1$ such that $a^l\equiv 1\ ({\rm mod}\ N)$.

\subsubsection{Invariance properties of residue maps in motivic cohomology}

\label{invpropres}

In this paragraph, we shall state certain elementary invariance properties of residue maps (see below for the definition of this term) 
in motivic cohomology. These invariance properties play a key role in the proof of Theorem \ref{polg}.

\begin{proposition}\label{torin-prop}
Let 
$$
\xymatrix{
X\ar@{^{(}->}[r]^i &Y\\
X_0\ar[u]^{f}\ar@{^{(}->}[r]^l & Y_0\ar[u]^{g}}
$$
be a cartesian diagram of smooth schemes over a field. Suppose that the horizontal 
morphisms are closed immersions. Let $c\in\bbN$ and suppose that the codimension 
of $X$ in $Y$ (resp. $X_0$ in $Y_0$) is $c$. Let $U:=Y\backslash X$ and $U_0:=Y_0\backslash X_0$. Then 
we have a commutative diagram of localisation sequences
$$
\xymatrix{
\dots\ar[r]&H_\sM^{\bullet-2c}(X,\ast-c)\ar[r]\ar[d] & H^\bullet_\sM(Y,\ast)\ar[r]\ar[d] & H^\bullet_\sM(U,\ast)\ar[r]\ar[d] & H_\sM^{\bullet+1-2c}(X,\ast-c)\ar[r]\ar[d]&\dots\\
\dots\ar[r]&H_\sM^{\bullet-2c}(X_0,\ast-c)\ar[r] & H^\bullet_\sM(Y_0,\ast)\ar[r] & H^\bullet_\sM(U_0,\ast)\ar[r] & H_\sM^{\bullet+1-2c}(X_0,\ast-c)\ar[r]&\dots
}
$$
where the vertical maps are the pull-back maps.
\end{proposition}
\begin{proof} 
See \cite{Deglise-Exc}. 
\end{proof}

\begin{proposition}\label{rdres-prop}
Let $Y$ be a smooth scheme over $k$ and let $i:X\hookrightarrow Y$ be a smooth closed subscheme. 
Let $X_0$ be a smooth closed subscheme of $X$. Let $c_X,c_Y\in\bbN$. 
Suppose that the codimension of $X_0$ in $X$ (resp. in $Y$) is $c_X$ (resp. $c_Y$). Then the diagram
$$
\xymatrix{
H^{\bullet}_\sM(X\backslash X_0,\ast)\ar[d]^{i_*}\ar[r] & H^{\bullet+1-2c_X}(X_0,\ast-c_X)\ar[d]^{=}\\
H^{\bullet+2(c_Y-c_X)}_\sM(Y\backslash X_0,\ast+c_Y-c_X)\ar[r] & H^{\bullet+1-2c_X}(X_0,\ast-c_X)
}
$$
is commutative. Here the horizontal maps are the residue maps.
\end{proposition}
\begin{proof}
We leave this an an exercise for the reader. See \cite[par. 6]{Jannsen-LNM1400}.
\end{proof}

Here is how Proposition \ref{torin-prop} applies in our setting.

Let $T$ be a an arithmetic variety over $R=k$ and $T\to S$ be a morphism of schemes over $R$. Let 
$\cA_T$ be the abelian scheme obtained by base-change and let ${\rm BC}:\cA_T\to \cA$ be the corresponding 
morphism. 

Let $$\BC_l:{\mathcal A}_T[N]\backslash\epsilon_{\cA_T}(S)\to {\mathcal A}[N]\backslash\epsilon_\cA(S)$$
be the morphism obtained by restricting $\BC$. Similarly, let 
$$\BC_{h}:{\mathcal A}_T\backslash{\mathcal A}_T[N]\to {\mathcal A}\backslash{\mathcal A}[N]$$ 
be the morphism obtained by restricting $\BC$.

\begin{lemma}\label{rescom}
 We have $\BC^*_l\circ\rho_\cA=\rho_{\cA_T}\circ\BC^*_h$.
\end{lemma}
\begin{proof}
Follows from \ref{torin-prop}.  \end{proof}

\subsubsection{An intermediate result}
 
 \begin{proposition}
 {\rm (a)} The map $$H^{2g-1}_{\mathcal M}(\cA\backslash\cA[N],g)^{(0)}\to 
 H^{2g-1}_{D,\an}((\cA\backslash\cA[N])_\bbR,\bbR(g))$$ induced by $\cyc_\an$ is injective.

{\rm (b)}  Let $\sigma\in\cA[N](S)\backslash\epsilon_\cA(S)$. Let $\pol_{-\sigma}\in H^{2g_\cA-1}_{\mathcal M}(\cA\backslash\cA[N],g_\cA))^{(0)}$ be the element 
corresponding to the class of $-\sigma$ in $H^0_{\mathcal M}(\cA[N]\backslash\epsilon_\cA(S),0)$. Then $$-2\cdot \cyc_\an(\pol_{-\sigma})=(\sigma^*\mfg_{\cA^\vee}-\mfg_{\cA^\vee})|_{\cA\backslash\cA[N]}.$$
   \label{polgpr} 
 \end{proposition}
 
 The proof of Proposition \ref{polgpr} will rely on the following lemmata.
 
 \begin{lemma}
  Let $\sigma\in\cA[N](S)\backslash\epsilon_\cA(S)$. 
  Then $$(\sigma^*\mfg_{\cA^\vee}-\mfg_{\cA^\vee})|_{\cA\backslash\{\epsilon_\cA(S),-\sigma(S)\}}\subseteq \cyc_\an(H^{2g_\cA-1}_{\mathcal M}(\cA\backslash\{\epsilon_\cA(S),-\sigma(S)\},g_\cA))^{(0)})$$
   \label{eqal}
 \end{lemma}
 \begin{proof}
 It is sufficient to show that 
 \begin{equation}
 (\sigma^*\mfg_{\cA^\vee}-\mfg_{\cA^\vee})|_{\cA\backslash\{\epsilon_\cA(S),-\sigma(S)\}}\subseteq\cyc_\an(H^{2g_\cA-1}_{\mathcal M}(\cA\backslash\cA[N],g_\cA)).
 \label{intereq}
\end{equation}
Indeed, a natural analog of the operator $\tr_{[a]}$ operates on analytic Deligne cohomology 
and the map $\cyc_\an$ intertwines this operator with $\tr_{[a]}$. So if \ref{intereq} holds, 
we may deduce the conclusion of the lemma from the existence of the Jordan decomposition. 

Let $\Sigma$ be the rigidified line bundle on $\cA^\vee$ corresponding to 
$\sigma$ via the $S$-isomorphism $\cA\simeq(\cA^\vee)^\vee$. Let 
$\overline{\Sigma}$ be the unique hermitian line bundle on $\cA^\vee$, such that 

- its underlying line bundle is $\Sigma;$

- the rigidification is an isometry;

- its curvature form is translation invariant on the fibres of the map $\cA^\vee(\bbC)\to S(\bbC).$

To prove relation \ref{eqal}, we compute in $\ari{\rm CH}^g(\cA)_\bbQ$:
\begin{eqnarray*}
\sigma^*(S_0,\mfg_{{\mathcal A}^\vee})-(S_0,\mfg_{{\mathcal A}^\vee})&=&
(-1)^g\Big(\sigma^*(p_{1,*}(\ari{\ch}(\overline{{\mathcal P}}))^{(g)})-p_{1,*}(\ari{\ch}(\overline{{\mathcal P}}))^{(g)}\Big)\\
&=& 
(-1)^g\Big(p_{1,*}(\ari{\ch}(\overline{{\mathcal P}}\otimes\overline{\Sigma}))^{(g)}-p_{1,*}(\ari{\ch}(\overline{{\mathcal P}}))^{(g)}\Big)\\
&=&
(-1)^g\Big(p_{1,*}(\ari{\ch}(\overline{{\mathcal P}}))^{(g)}-p_{1,*}(\ari{\ch}(\overline{{\mathcal P}}))^{(g)}\Big)=0
\end{eqnarray*}
Here we used in the second equality the fact that the direct image in arithmetic Chow theory is naturally compatible 
with smooth base-change. We also used the fact that 
$\ari{\ch}(\overline{\Sigma})=1$ and the multiplicativity of the arithmetic Chern character. Now, we have 
$$
\sigma^*(S_0,\mfg_{{\mathcal A}^\vee})-(S_0,\mfg_{{\mathcal A}^\vee})=
(-\sigma(S),\sigma^*\mfg_{{\mathcal A}^\vee})-(S_0,\mfg_{{\mathcal A}^\vee})
$$
and thus the image of 
$$
\big(\sigma^*\mfg_{\cA^\vee}-\mfg_{\cA^\vee}\big)|_{\cA\backslash\{\epsilon_\cA(S),-\sigma(S)\}}
$$
in $\ari{\rm CH}^g(\cA\backslash\{\epsilon_\cA(S),-\sigma(S)\})_\bbQ$ vanishes. Using \ref{arfundseq}, we 
may conclude. 
 \end{proof}
 
 The next lemma proves assertion (a) in Proposition \ref{polgpr}.
 \begin{lemma}
The map $H^{2g-1}_{\mathcal M}(\cA\backslash\cA[N],g))^{(0)}\to 
 H^{2g-1}_{D,\an}((\cA\backslash\cA[N])_\bbR,\bbR(g))$ induced by $\cyc_\an$ is injective.
 \label{injlem}
 \end{lemma}
 \begin{proof}
 Notice that we have commutative diagram
 $$
\xymatrix{
H^{2g-1}_\sM({\mathcal A}\backslash{\mathcal A}[N],g)^{(0)}\ar[r]^{\rho_\cA,\isom}\ar[d]^{\cyc_\an} & 
H^0_\sM({\mathcal A}[N]\backslash\epsilon_\cA(S),0)\ar[d]^{\cyc_\an}\\
H^{2g-1}_{D,\an}({\mathcal A}\backslash{\mathcal A}[N],g)\ar[d]^{\phi_B} & 
H^0_{D,\an}({\mathcal A}[N]\backslash\epsilon_\cA(S),0)\ar[d]^{\phi_B,\isom}\\
H^{2g-1}({\mathcal A}(\bbC)\backslash{\mathcal A}[N](\bbC),g)\ar[r]^{\isom} &
H^0({\mathcal A}[N](\bbC)\backslash\epsilon_\cA(S)(\bbC),0)
}
$$
where the isomorphism on the bottom line is the residue map in Betti 
cohomology. The fact that the residue map in Betti 
cohomology is an isomorphism 
follows from the fact that Betti cohomology is a twisted Poincar\'e duality 
theory in the sense of Bloch-Ogus and Corollary \ref{corfund}. 
Furthermore, the upper right vertical map can be seen to be injective. This implies the assertion of the Lemma.
\end{proof}

\begin{lemma}
Let $\cB\to S$ be an abelian scheme. 
%Suppose that $$\cA[N]\simeq(\bbZ/N\bbZ)^{2g}_S$$ and $$\cB[N]\simeq(\bbZ/N\bbZ)^{2g}_S.$$ 

Let $\sigma\in\cA[N](S)$ and let 
$\tau\in \cB[N](S)$. 

Let $x:=(\sigma^*\mfg_{\cA^\vee}-\mfg_{\cA^\vee})|_{\cA\backslash\cA[N]}$ and $y:=(\tau^*\mfg_{\cB^\vee}-\mfg_{\cB^\vee})|_{\cB\backslash\cB[N]}$. 

Furthermore, let $$z:=\big((\sigma\times\tau)^*\mfg_{\cA^\vee\times\cB^\vee}-\mfg_{\cA^\vee\times\cB^\vee}\big)|_{\cA\times\cB\backslash(\cA\times\cB)[N]}.$$ Then we have
$$
z=(\Id_{\cA\backslash\cA[N]}\times\tau)_*x+(0\times\Id_{\cB\backslash\cB[N]})_*y.
$$
\label{green-lemma}
\end{lemma}
\begin{proof}
Let $\nu_{\cA^\vee}:=(-1)^g p_{1,*}({\ch}(\overline{{\mathcal P}}))^{(g)}$.  We 
shall write $\delta_{\cA\times 0}$ for the Dirac current in \mbox{$\cA(\bbC)\times_{S(\bbC)}\cB(\bbC)$} of 
the closed complex manifold associated with the image of $\cA$ by the morphism $(\Id_{\cA}\times_S\epsilon_{\cB}\circ \pi_{\cA}).$ 
We shall also use the notation $\delta_{\cA\times\tau}$, which is defined similarly.
Let $q_\cA:\cA\times_S\cB\to\cA$ and $q_\cB:\cA\times_S\cB\to\cB$ be the obvious projections. 
First we make the computation
\begin{eqnarray*}
&&q_\cA^*\mfg_{\cA^\vee}\wedge(\delta_{\cA\times 0}-\delta_{\cA\times\tau})
+
(q_\cB^*\tau^*\mfg_{\cB^\vee}-q_\cB^*\mfg_{\cB^\vee})\wedge(\delta_{0\times\cB}-q_{\cA}^*\nu_{\cA^\vee})\\
&=&
q_\cA^*\mfg_{\cA^\vee}\wedge(\delta_{\cA\times 0}-\delta_{\cA\times\tau})
-
(q_\cB^*\tau^*\mfg_{\cB^\vee}-q_\cB^*\mfg_{\cB^\vee})\wedge({\rm dd}^cq_{\cA}^*\mfg_{\cA^\vee})\\
&=&
(\delta_{\cA\times 0}-\delta_{\cA\times\tau}+\delta_{\cA\times\tau}-q^*_\cB\tau^*\nu_{\cB^\vee}-\delta_{\cA\times 0}+q^*_\cB\nu_{\cB^\vee})\wedge q_\cA^*\mfg_{\cA^\vee}=0
\end{eqnarray*}
In this computation, we used the following elementary fact (see 
\cite[before par. 6.2.1]{Roessler-Adams}): if $\eta,\omega$ are two currents of 
type $(p,p)$ on a complex manifold, such that $\eta$ and $\omega$ have disjoint wave front sets, 
then $$\eta\wedge{\rm dd}^c\omega-{\rm dd}^c\eta\wedge\omega\in{\rm im}\partial+{\rm im}\bar\partial.$$
Now using 
the formula \cite[Th. 1.2, 5.]{Loopaper} for the canonical current of fibre products of abelian schemes and the previous computation, we get that 
\begin{eqnarray*}
&&(\sigma\times\tau)^*\mfg_{\cA^\vee\times\cB^\vee}-\mfg_{\cA^\vee\times\cB^\vee}\\
&=&
(\sigma\times\tau)^*\big(q_\cA^*\mfg_{\cA^\vee}\ast q_{\cB}^*\mfg_{\cB^\vee}\big)-\mfg_{\cA^\vee\times\cB^\vee}=
q_\cA^*\sigma^*\mfg_{\cA^\vee}\ast q_{\cB}^*\tau^*\mfg_{\cB^\vee}-q_\cA^*\mfg_{\cA^\vee}\ast q_{\cB}^*\mfg_{\cB^\vee}\\
&=&
q_\cA^*\sigma^*\mfg_{\cA^\vee}\wedge\delta_{\cA\times\tau}+q_{\cA}^*\sigma^*\nu_{\cA^\vee}\wedge 
q_{\cB}^*\tau^*\mfg_{\cB^\vee}-
q_\cA^*\mfg_{\cA^\vee}\wedge\delta_{A\times 0}-q_\cA^*\nu_{\cA^\vee}\wedge q_\cB^*\mfg_{\cB^\vee}\\
&=&
q_\cA^*\sigma^*\mfg_{\cA^\vee}\wedge\delta_{\cA\times\tau}+q_{\cA}^*\sigma^*\nu_{\cA^\vee}\wedge 
q_{\cB}^*\tau^*\mfg_{\cB^\vee}-
q_\cA^*\mfg_{\cA^\vee}\wedge\delta_{A\times 0}-q_\cA^*\nu_{\cA^\vee}\wedge q_\cB^*\mfg_{\cB^\vee}\\
&+&q_\cA^*\mfg_{\cA^\vee}\wedge(\delta_{\cA\times 0}-\delta_{\cA\times\tau})
+
(q_\cB^*\tau^*\mfg_{\cB^\vee}-q_\cB^*\mfg_{\cB^\vee})\wedge(\delta_{0\times\cB}-q_{\cA}^*\nu_{\cA^\vee})\\
&=&(q_\cA^*\sigma^*\mfg_{\cA^\vee}-
q_\cA^*\mfg_{\cA^\vee})\wedge\delta_{\cA\times \tau}+ 
(q_{\cB}^*\tau^*\mfg_{\cB^\vee}-q_\cB^*\mfg_{\cB^\vee})\wedge\delta_{0\times\cB}
\end{eqnarray*}
%Here the product $(\cdot)\ast(\cdot)$ is the star product of Green currents (see ...). Now notice that 
%the cycle class of $\cA\times\tau$ in $H^{2g_\cB}_{D,\an}((\cA\times\cB)_\bbR,\bbR(g_\cB))$ and the cycle class  of 
%$\cA\times0$ in $H^{2g_\cB}_{D,\an}((\cA\times\cB)_\bbR,\bbR(g_\cB))$ are the same, because 
%$\cA\times\tau$ and $\cA\times 0$ are algebraically equivalent. Also, ... and ... imply that 
%$\sigma^*\nu_{\cA^\vee}=\nu_{\cA^\vee}$ and ... implies that the class of $\nu_{\cA^\vee}$ in 
%$H_{D,\an}^{2g_\cA}(\cA,g_\cA)$ is equal to the cycle class of the $\epsilon_\cA(S)$. 
%So we have
%\begin{eqnarray*}
%&&\Big((\sigma\times\tau)^*\mfg_{\cA^\vee\times\cB^\vee}-\mfg_{\cA^\vee\times\cB^\vee}\Big)|_{\cA\times\cB\backslash 
%(\cA\times\cB)[N]}\\
%&=&
%\Big((\pi_\cA^*\sigma^*\mfg_{\cA^\vee}-
%\pi_\cA^*\mfg_{\cA^\vee})\wedge\delta_{\cA\times 0}\Big)|_{\cA\times\cB\backslash 
%(\cA\times\cB)[N]}+ 
%\Big((\pi_{\cB}^*\tau^*\mfg_{\cB^\vee}-\pi_\cB^*\mfg_{\cB^\vee})\wedge\delta_{0\times\cB}\Big)|_{\cA\times\cB\backslash 
%(\cA\times\cB)[N]}\\
%&=&
%\Big((\pi_\cA^*\sigma^*\mfg_{\cA^\vee}-
%\pi_\cA^*\mfg_{\cA^\vee})\wedge\delta_{\cA\times \tau}\Big)|_{\cA\times\cB\backslash 
%(\cA\times\cB)[N]}+ 
%\Big((\pi_{\cB}^*\tau^*\mfg_{\cB^\vee}-\pi_\cB^*\mfg_{\cB^\vee})\wedge\delta_{0\times\cB}\Big)|_{\cA\times\cB\backslash 
%(\cA\times\cB)[N]}
%\end{eqnarray*}
which implies the assertion of the lemma. 
\end{proof}

\begin{lemma}\label{mot-lem}
Let $\cB\to S$ be an abelian scheme of relative dimension $g_\cB$. 
%Suppose that $$\cA[N]\simeq(\bbZ/N\bbZ)^{2g}_S$$ and $$\cB[N]\simeq(\bbZ/N\bbZ)^{2g}_S.$$ 
Let $\sigma\in\cA[N](S)\backslash\epsilon_\cA(S)$ and let 
$\tau\in \cB[N](S)\backslash\epsilon_\cB(S)$. 

Let $x\in H^{2g_\cA-1}_{\mathcal M}(\cA\backslash\cA[N],g_\cA))^{(0)}$ be the element 
corresponding to the class of $\sigma$ in $H^0_{\mathcal M}(\cA[N]\backslash\epsilon_\cA(S),0)$. 

Let $y\in H^{2g_\cB-1}_{\mathcal M}(\cB\backslash\cB[N],g_\cB))^{(0)}$ be the element 
corresponding to the class of $\tau$ in  $H^0_{\mathcal M}(\cB[N]\backslash\epsilon_\cB(S),0)$.

Then $z:=(\Id_{\cA\backslash\cA[N]}\times \tau)_*x+(0\times\Id_{\cB\backslash\cB[N]})_*y$ lies in $H^{2(g_\cA+g_\cB)-1}_{\mathcal M}(\cA\times_S\cB\backslash(\cA\times_S\cB)[N],g_\cA+g_\cB))^{(0)}$ and 
corresponds to the class of $\sigma\times\tau$ in $H^0_{\mathcal M}((\cA\times_S\cB)[N]\backslash\epsilon_{\cA\times_S\cB}(S),0)$. 
\end{lemma}
\begin{proof} For the first assertion, it is sufficient to show that $$z\in H^{2(g_\cA+g_\cB)-1}_{\mathcal M}(\cA\times_S\cB\backslash(\cA\times_S\cB)[N],g_\cA+g_\cB))^{(0)}$$ and that $z$ has residue $\sigma\times\tau$ in \mbox{$H^0_{\mathcal M}((\cA\times\cB)[N]\backslash\epsilon_{\cA\times\cB}(S),0)$.} 

The fact that  $$z\in H^{2(g_\cA+g_\cB)-1}_{\mathcal M}(\cA\times_S\cB\backslash(\cA\times_S\cB)[N],g_\cA+g_\cB))^{(0)}$$ 
follows from the fact that the morphisms $\Id_{\cA\backslash\cA[N]}\times \tau$ and 
$0\times\Id_{\cB\backslash\cB[N]}$ commute with the multiplication-by-$a$ morphism (since $a\equiv 1\bmod N$). 

To compute the 
residue of $z$, notice first that the value of the residue of $x$ (resp. $y$) at $\epsilon_\cA(S)$ 
(resp. $\epsilon_\cB(S)$) is $-1$. To see this, notice that the push-forward into $H^{2g_\cA}_{\mathcal M}(\cA,g_\cA)$ (resp. into $H^{2g_\cB}_{\mathcal M}(\cB,g_\cB)$) of the residue of $x$ (resp. $y$)  in $H^0(\cA[N],0)$ 
(resp. $H^0(\cB[N],0)$) vanishes. This vanishing implies that the push-forward to 
$H^0(S,0)\simeq\bbQ$ of the residue of 
$x$ (resp. $y$)  in $H^0(\cA[N],0)$ 
(resp. $H^0(\cB[N],0)$) vanishes too. If one combine this fact with the fact 
that the residue of 
$x$ (resp. $y$)  in $H^0(\cA[N]\backslash\epsilon_\cA(S),0)$ 
(resp. $H^0(\cB[N]\backslash\epsilon_\cB(S),0)$) is $\sigma(S)$ (resp. $\tau(S)$) (this holds by hypothesis), one obtains 
that the residue of $x$ (resp. $y$) at $\epsilon_\cA(S)$ 
(resp. $\epsilon_\cB(S)$) is $-1$.

Now using Proposition 
\ref{rdres-prop}, we may compute that the residue of $z$ in $H^0((\cA\times_S\cB)[N],0)$ is 
$$
(\sigma\times\tau-0\times \tau)+(0\times\tau-0\times 0)
$$
and thus the residue of $z$ in \mbox{$H^0_{\mathcal M}((\cA\times\cB)[N]\backslash\epsilon_{\cA\times\cB}(S),0)$}  is $\sigma\times\tau$.\end{proof}

\begin{lemma}
Proposition \ref{polgpr} (b) holds if the morphism $S\to\Spec\ k$ is the identity and \mbox{$\cA\simeq \prod_{i=1,S}^g E_i$,} where 
$E_i$ is an elliptic curve over $S=\Spec\ k.$ 
\label{ellpart}
\end{lemma}
\begin{proof}(of Lemma \ref{ellpart}) 
We first prove the statement when $g=1$. Let $E:=E_1=E_g$. Notice that the map 
\begin{eqnarray*}
&\cyc_\an&:H^1_\sM(E\backslash E[N],1)={\mathcal O}^*_{E\backslash E[N]}(E\backslash E[N])\otimes\bbQ\\
&\to& H^1_{D,\an}(E\backslash E[N],1)=\{f\in C^\infty((E\backslash E[N])(\bbC),\bbR)\ |\ {\rm dd}^cf=0\}
\end{eqnarray*} can be explicitly described by the formula $u\otimes r\mapsto r\log|u(\bbC)|$. 

Now let $\alpha$ be one of the given embeddings of $R$ into $\bbC$ (those that are part of the datum of an 
arithmetic ring).  Let 
$\bbC/[1,\tau_{E,\alpha}]\simeq E(\bbC)_\alpha$ be a presentation of $E(\bbC)_\alpha$ as a quotient 
of $\bbC$ by a lattice generated by $1$ and a complex number $\tau_{E,\alpha}$ with strictly positive imaginary part. 
Call $\lambda:\bbC\to E(\bbC)_\alpha$ the corresponding quotient map. 
Then by \cite[par. 7]{Loopaper}, we have
$$
\mfg_{E^\vee}(\lambda(z))=-2\log|e^{-z\cdot\eta(z)/2}{\rm sigma}(z)\Delta(\tau_{E,\alpha})^{1\over 12}|
$$
for all $z\not\in[1,\tau_{E,\alpha}]$. Here $\Delta(\bullet)$ is the discriminant modular form, ${\rm sigma}(z)$ is the Weierstrass sigma-function associated with 
the lattice $[1,\tau_{E,\alpha}]$ and $\eta$ is the quasi-period map associated with 
the lattice $[1,\tau_{E,\alpha}]$, extended $\bbR$-linearly to 
all of $\bbC$ (see \cite[I, Prop. 5.2]{Silverman-Advanced} for the latter). Let $z_0\in\bbC$ such that 
$\lambda(z_0)=-\sigma_\alpha$. We compute 
% z_0\mapsto -z_0 %%%%%%%%%%%%%%%%%%%%%%%%%%
\begin{eqnarray*}
&&|{e^{-(z-z_0)\cdot\eta(z-z_0)/2}{\rm sigma}(z-z_0)\over 
e^{-z\cdot\eta(z)/2}{\rm sigma}(z)}|\\
&=&
|{e^{-{1\over 2}(z\cdot\eta(z)-z\cdot\eta(z_0)-z_0\cdot\eta(z)+z_0\cdot\eta(z_0))}{\rm sigma}(z-z_0)\over 
e^{-z\cdot\eta(z)/2}{\rm sigma}(z)}|\\
&=&
|e^{{1\over 2}(z\cdot\eta(z_0)+z_0\cdot\eta(z)-z_0\cdot\eta(z_0))}{
{\rm sigma}(z-z_0)\over 
{\rm sigma}(z)}|\\
&=&
|e^{z\cdot\eta(z_0)-{1\over 2}z_0\cdot\eta(z_0)}{
{\rm sigma}(z-z_0)\over 
{\rm sigma}(z)}|.
\end{eqnarray*}
Here we used the Legendre relation for the quasi-period map on the last line (see 
\cite[I, Prop. 5.2, (d)]{Silverman-Advanced}). 
Let 
$$
\phi(z):=e^{z\cdot\eta(z_0)-{1\over 2}z_0\cdot\eta(z_0)}{
{\rm sigma}(z-z_0)\over 
{\rm sigma}(z)}.
$$
Now recall that the periodicity relation for the sigma function says that 
$$
{\rm sigma}(z+\omega)=\psi(\omega)e^{\eta(\omega)(z+\omega/2)}{\rm sigma}(z)
$$
for all $\omega\in [1,\tau_{E,\alpha}]$ (see \cite[I, Prop. 5.4 (c)]{Silverman-Advanced}). Here $\psi(\bullet)$ is a function with values in 
the set $\{-1,1\}$. This implies that 
$$
{{\rm sigma}(z+\omega-z_0)\over 
{\rm sigma}(z+\omega)}={{\rm sigma}(z-z_0)\over 
{\rm sigma}(z)}e^{-\eta(\omega)z_0}
$$
and thus that $$\phi(z+\omega)/\phi(z)=e^{\omega\cdot\eta(z_0)-\eta(\omega)z_0}:=\alpha(\omega,z_0).$$ The Legendre relation again implies that 
$\alpha(\bullet,z_0)$ defines a homomorphism of abelian groups $[1,\tau_{E,\alpha}]\to 
\bbC^*$ and that its image is a torsion group of order dividing $N$. 
We conclude that $\phi(z)^{N}$ is a $[1,\tau_{E,\alpha}]$-periodic function. 
Furthermore, $\phi(z)$ has a zero of order $1$ at $z_0$ and 
a pole of order $1$ at $0$. We see that after passage to the quotient, the function $\phi(z)$ 
defines an element $\phi_0\in {\mathcal O}^*(E(\bbC)_\alpha\backslash E(\bbC)_\alpha[N])\otimes\bbQ$,  whose divisor is $z_0\ominus 0$. Furthermore, the distribution relations of A. Robert (see \cite[par. 4, Th. 4.1]{Kubert-Lang-Modular}) 
show that $\tr_{[a]}(\phi_0)=\phi_0$ (we slighty abuse notation here). 

Now let  
$\widetilde{\phi}_0\in H^1_\sM(E\backslash E[N],1)^{(0)}$ be an element 
such that $$\cyc_\an(\widetilde{\phi}_0)|_{E(\bbC)_\alpha}=\log|\widetilde{\phi}_0|=
([-z_0]^*\mfg_{E^\vee}-\mfg_{E^\vee})|_{E(\bbC)_\alpha\backslash E(\bbC)_\alpha[N]}=-2\log|\phi_0|=\log|\phi_0^{-2}|$$ This exists by Lemma \ref{eqal}. Notice that both $\widetilde{\phi}_0$ and $\phi_0$ are invariant under $\tr_{[a]}$ and thus ${\widetilde{\phi}_0(\bbC)\over\phi_0^{-2}}$ 
is a constant $l\in\bbC^*\otimes\bbQ$, such that $\tr_{[a]}(l)=l^{a^2}=l$ and thus $l=1$.
 Thus we may compute that $\div(\widetilde{\phi}_0)=2\odot 0\ominus 2\odot z_0$. We conclude 
 that $\widetilde{\phi}_0=-2\cdot \pol_{z_0}$ and this concludes the proof of the lemma when $g=1.$ 
To prove the Lemma \ref{ellpart} in general, just combine the fact that Lemma 
\ref{ellpart} holds for 
$g=1$ with Lemmata \ref{green-lemma} and \ref{mot-lem}.

%%%%%%%%%%%%%%%%%%%%%%%%%%%%%%%%%%
%We now turn to the proof of (b). Let 
%$z\in H^{2g-1}_{\mathcal M}(\cA\backslash\cA[N],g))^{(0)}$. According to Lemma \ref{mot-lem}, there 
%are elements $z_i\in H^{1}_{\mathcal M}(E_{i}\backslash E_{i}[N],1))^{(0)}$ such that 
%$$
%z=\sum_{i=1}^g(\tau_{1,i}\times\dots \tau_{i-1,i}\times\Id_{E_i\backslash E_i[N]}\times \tau_{i+1,i}\dots\times \tau_{g,i})_*z_i
%$$
%where the $\Id_{E_i\backslash E_i[N]}$ term lies in the $i$-th slot. Let now $P_i\in E_i(\bbC)$ be such that $P_i\not\in E_{i,\bbC}[N](\bbC)$. Using Lemma \ref{green-lemma}, we compute in analytic Deligne cohomology 
%\begin{eqnarray*}
%&&(\tau_{1,i,\bbC}\times\dots\times \tau_{i-1,i,\bbC}\times\Id_{E_{i-1,\bbC}\backslash E_{i-1,\bbC}[N]}\times P_i\times 
%\tau_{i+1,i,\bbC}\dots\times \tau_{g,i,\bbC})^*
%\cyc_\an(z)\\
%&=&P_i^*(\cyc_\an(z_{i}))\cdot\delta_{P_i})\in H^{3}_{D,\an}(E_{i},\bbR(2))
%\end{eqnarray*}
%and
%$$
%\int_{E_i(\bbC)}P_i^*(\cyc_\an(z_{i}))\cdot\delta_{P_i}=P_i^*(\cyc_\an(z_{i}))
%$$
%Since $P_i\in E_i(\bbC)$ was arbitrary, this shows that if $\cyc_\an(z)=0$ then 
%$z_i=0$ for all $i$ and hence $z=0$. 
\end{proof}

Let $T$ be a an arithmetic variety over $R=k$ and $T\to S$ be a morphism of schemes over $R$. Let 
$\cA_T$ be the abelian scheme obtained by base-change and let ${\rm BC}, \BC_h$ and $\BC_l$ be as at the end of paragraph \ref{invpropres}. Let 
$$\BC_l^*:H_\sM^0(\cA[N]\backslash 
\epsilon_\cA(S),0)\to H_\sM^0(\cA_T[N]\backslash 
\epsilon_{\cA_T}(S),0)$$ and $$\BC_h^*:H^{2g-1}_{\sM}({\mathcal A}\backslash{\mathcal A}[N],g)^{(0)}\to H^{2g-1}_{\sM}({\mathcal A}_T\backslash{\mathcal A}_T[N],g)^{(0)}$$
be the natural pull-back maps.

\begin{lemma}
If $T$ is connected and $\cA[N]\simeq(\bbZ/N\bbZ)^{2g}_S$ then the maps $\BC_l^*$ and $\BC_h^*$
 are isomorphisms.
\label{isolem}
\end{lemma}
\begin{proof}
The map $\BC_l^*$ is an isomorphism by construction and the fact that the map $\BC_h^*$ is an isomorphism 
follows from the fact that $\BC_l^*$ is an isomorphism, from Lemma \ref{rescom} and from the isomorphisms
 $$
\rho_\cA:H^{2g-1}_\sM({\mathcal A}\backslash{\mathcal A}[N],g)^{(0)}\isom
H^0_\sM({\mathcal A}[N]\backslash\epsilon_{\cA}(S),0).
$$
and 
 $$
\rho_{\cA_T}:H^{2g-1}_\sM({\mathcal A}_T\backslash{\mathcal A}_T[N],g)^{(0)}\isom
H^0_\sM({\mathcal A}_T[N]\backslash\epsilon_{\cA_T}(S),0).
$$
\end{proof}

\begin{lemma} 
 Let $\sigma\in\cA[N](S)\backslash\epsilon_\cA(S)$. We have
$$
\BC^*_h\pol_\sigma=\pol_{\sigma_T}
$$
and
$$
\BC^*_h\big(\sigma^*\mfg_{\cA^\vee}-\mfg_{\cA^\vee})|_{{\mathcal A}\backslash{\mathcal A}[N]}\big)=(\sigma_T^*\mfg_{\cA^\vee_T}-\mfg_{\cA^\vee_T})|_{{\mathcal A}_T\backslash{\mathcal A}_T[N]}.
$$
\label{BClem}
\end{lemma}
\begin{proof}
The second equality is a consequence of the remarks after Theorem \ref{mainth0}. The 
first  equality is 
a consequence of the fact that $\BC_h^*\circ\tr_{a,\cA}=\tr_{a,\cA_T}\circ\BC_h^*$ and 
of Lemma \ref{rescom}.
% Quillen LNM 341
\end{proof}
 
\begin{lemma}\label{implem}
Suppose that  $\cA[N]\simeq(\bbZ/N\bbZ)^{2g}_S$. Suppose also that for 
some point $s\in S(k)$, the abelian scheme $\cA_{s}$ is a product of 
elliptic curves over $k$. Then Proposition \ref{polgpr} holds for $\cA\to S$. 
\end{lemma}
\begin{proof}
Let $T:=\Spec\ k$ and let $T\to S$ be the closed immersion given by $s$. By Lemmata \ref{injlem} and \ref{isolem},  the map
$$
\cyc_\an\circ \BC_h^*=\BC_h^*\circ\cyc_\an:H^{2g-1}_{\mathcal M}(\cA\backslash\cA[N],g))^{(0)}\to H^{2g-1}_{D,\an}((\cA_s\backslash\cA_s[N])_\bbR,\bbR(g))
$$
is injective. Now let $\phi\in H^{2g-1}_{\mathcal M}(\cA\backslash\cA[N],g))^{(0)}$ be an element 
such that $$-\cyc_\an(\phi)=(\sigma^*\mfg_{\cA^\vee}-\mfg_{\cA^\vee})|_{\cA\backslash\cA[N]}.$$
The element $\phi$ exists by Lemma \ref{eqal}. Then by Lemma \ref{BClem} and Lemma \ref{ellpart}, 
we have $$-\BC_h^*(\cyc_\an(\phi-\pol_{-\sigma}))=0$$ whence $\phi=\pol_{-\sigma}$. 
\end{proof}

\begin{lemma}\label{consttor}
Let $M\in\bbN^*$. 
To prove Proposition \ref{polgpr}, it is sufficient to prove it under the supplementary assumption that 
$\cA[M]\simeq(\bbZ/M\bbZ)^{2g}_S$.
\end{lemma}
\begin{proof}
We may choose $T\to S$, such that $T\to S$ is proper and generically finite and such that 
$\cA_T[M]\simeq(\bbZ/M\bbZ)^{2g}_T$. In view of this as well as Lemma \ref{BClem} and Lemma 
\ref{eqal}, we see that 
it is sufficient to show that the map $\BC_h^*$ 
is injective when $T\to S$ is proper and generically finite. This is a consequence of the fact that 
$\BC_h$ is then also proper and generically finite, of the projection formula, and 
of the fact that $\BC_{h,*}(1)=\deg(\BC_h)$. 
\end{proof}

{\it Proof of Proposition \ref{polgpr}.} 
 
Let $M\in\bbN^*$. Suppose that $N|M$ and that $M>4$. By Lemma \ref{consttor}, 
we may suppose without restriction of generality that $\cA[M]\simeq (\bbZ/M\bbZ)^{2g}_S$. Now 
by Lemma \ref{BClem}, to prove Proposition \ref{polgpr}, it 
is sufficient to show that there exists 

- $T$, an arithmetic variety over $R$;

- $\cB\to T$, an abelian scheme of relative dimension $g$ such that $\cB[M]\simeq (\bbZ/M\bbZ)^{2g}_T$ and such that Proposition \ref{polgpr} holds for $\cB$; 

- an $R$-morphism $S\to T$, such that $\cA\simeq \cB\times_T S$.

Before trying to determine a suitable $\cB$, equip $\cA\to S$ with a polarisation 
and let $d$ be the degree of the latter. 
As a candidate for $\cB$, we consider the universal abelian scheme over the (fine) moduli space ${\rm A}_{g,d,M}/k$ of abelian varieties of dimension $g$, equipped with an $M$-level structure and a polarisation of 
degree $d$ (see for instance \cite{Moret-Bailly-Pinceaux} for more details). Proposition \ref{polgpr} holds for $\cB$ by Lemma \ref{implem} and there exists a $k$-morphism $S\to T$, such $\cA\simeq \cB\times_T S$, because ${\rm A}_{g,d,M}/k$ is a moduli space.

\subsubsection{Proof of Theorem \ref{polg}} 

\label{ssspolg}

\begin{lemma}\label{consttorth}
Let $M\in\bbN^*$. 
To prove Theorem \ref{polg}, it is sufficient to prove it under the supplementary assumption that 
$\cA[M]\simeq(\bbZ/M\bbZ)^{2g}_S$.
\end{lemma}
\begin{proof}
The proof is similar to the proof of Lemma \ref{consttor} and will be omitted.
\end{proof}
\begin{lemma}
Suppose that $\cA[N]\simeq(\bbZ/N\bbZ)^{2g}_S$. For any $\sigma\in \cA[N](S)$, let 
$$x_\sigma:=\sigma^*\mfg_{\cA^\vee}-\mfg_{\cA^\vee}.$$ Then we have
$$\sum_\sigma x_\sigma=[N]^*\mfg_{\cA^\vee}-N^{2g}\cdot \mfg_{\cA^\vee}.$$
\label{prothlink}
\end{lemma}
\begin{proof}
If $\eta$ is a differential form on $\cA(\bbC)$, we have from the definitions that 
$$
[N]^*[N]_*\eta=\sum_\sigma\sigma_*\eta
$$
Dualising this statement for currents, we obtain that 
$$
[N]^*[N]_*\mfg_{\cA^\vee}=\sum_\sigma\sigma^*\mfg_{\cA^\vee}.
$$
Since $[N]_*\mfg_{\cA^\vee}=\mfg_{\cA^\vee}$ by \ref{mainth0} (c), we obtain 
that 
$$
[N]^*\mfg_{\cA^\vee}=\sum_\sigma\sigma^*\mfg_{\cA^\vee}
$$
from which the Lemma follows. 
\end{proof}
The proof of Theorem \ref{polg} now follows from Lemma 
\ref{consttorth}, Proposition \ref{polgpr} and Lemma \ref{prothlink}.

\bibliographystyle{alpha}

\end{document}